\documentclass[11pt,a4paper,twoside]{article}
\usepackage{geometry} 
\usepackage{bbm,mathrsfs,eurosym} 
\usepackage{amsmath, amsthm, amsfonts} 
\usepackage{color}

\newtheorem{thm}{Theorem}
\newtheorem{cor}{Corollary}
\newtheorem{lem}{Lemma}

\theoremstyle{remark}
\newtheorem{rem}{Remark}

\theoremstyle{remark}
\newtheorem{exa}{Example}

\theoremstyle{definition}
\newtheorem{defi}{Definition}

\newcommand{\dx}[0]{\,\mathrm{d}}

\newcommand{\fbsde}[0]{\textrm{}}
\newcommand{\esssup}[0]{\mathrm{ess}\,\mathrm{sup}}

\title{Existence, Uniqueness and Regularity of Decoupling Fields to Multidimensional Fully Coupled FBSDEs}
\author{Alexander Fromm, Peter Imkeller \vspace{3mm} \\ Institut f\"ur Mathematik \\ Humboldt-Universit\"at zu Berlin \\ Unter den Linden 6 \\ 10099 Berlin \\ Germany}
\date{\today}

\begin{document}

\maketitle

\begin{abstract}
We develop an existence, uniqueness and regularity theory for general multidimensional strongly coupled FBSDE using so called decoupling fields. We begin with a local result and extend it to a global theory via concatenation. The cornerstone of the global theory is the so called maximal interval which is, roughly speaking, the largest interval on which reasonable solutions exist. A method to verify that the maximal interval is the whole interval, for problems in which this is conjectured, is proposed. As part of our study of the regularity of solutions constructed we show variational differentiability under Lipschitz assumptions. \\ Extra emphasis is put on the more special Markovian case in which assumptions on the Lipschitz continuity for the FBSDE can be weakened to local ones, and additional regularity properties emerge.
\end{abstract}

\section{Introduction}

In recent decades Backward Stochastic Differential Equations (BSDE) and more generally Forward Backward Stochastic Differential Equations (FBSDE) have been studied extensively. They have many applications in various fields of applied mathematics, such as stochastic control theory and mathematical finance, and are closely tied to a major class of partial differential equations. A general FBSDE is a system of the form
$$ X_t=x+\int_{0}^t\mu(s,X_s,Y_s,Z_s)\dx s+\int_{0}^t\sigma(s,X_s,Y_s,Z_s)\dx W_s, $$
$$ Y_t=\xi(X_T)-\int_{t}^{T}f(s,X_s,Y_s,Z_s)\dx s-\int_{t}^{T}Z_s\dx W_s, $$
where $X$ and $Y$ can be multidimensional, such that the two equations represent systems of equations in general. The nature of the underlying problem is encoded in the parameter functions $\mu,\sigma,f$ which can be random, but at least progressively measurable, and the terminal condition $\xi$, which can also depend on $\omega$, but is required to be measurable w.r.t. $\mathcal{F}_T$, the information available at terminal time $T$.

The system is called decoupled if either $\mu,\sigma$ do not depend on $Y,Z$, or if $\xi$ does not depend on $X$. In these two cases the problem can be treated by solving one of the equations first, and then simply plugging the solution processes obtained into the other equation, in order to solve the latter in the second step. In both steps solutions can be constructed through a Picard iteration applying Banach's fixed point theorem. The theory of decoupled problems is much more extensive than the theory of general strongly coupled FBSDE, in which still many questions remain unanswered.

Several methods have been proposed to study these problems. The so called \emph{Four Step Scheme} (see \cite{mayong}) is based on reducing the problem to a quasi-linear parabolic PDE. This works for parameter functions which are deterministic and sufficiently smooth.

The \emph{Method of Continuation} (e.g. \cite{hupeng}) is purely stochastic, but relies on monotonicity assumptions for the parameter functions that might be hard to verify.

In order to develop a general technique, Zhang et al. have introduced the concept of \emph{Decoupling Fields} in \cite{jfzhang}. Decoupling fields can be used to extend the \emph{Contraction Method} proposed by Antonelli \cite{antonelli} to construct solutions on large intervals by patching together solutions constructed on small intervals. In \cite{jfzhang} the emphasis is on well-posedness: the authors are primarily interested in problems which have solutions on the whole interval $[0,T]$ and propose sufficient conditions based on the so called \emph{characteristic BSDE} and the \emph{dominating ODE}. Furthermore they concentrate on one-dimensional problems. We have found this to be a serious draw-back, since most coupled problems we have encountered involve multidimensional parameters, either in $W$ or $X$ or $Y$.

In this paper we will drop the restriction of well-posedness and develop an existence, uniqueness and regularity theory for the general case. In other words we consider ill-posed problems no less interesting or worthy of a rigorous study than well-posed problems. We will merely require Lipschitz continuity of the parameters in the general non-Markovian case and a form of local Lipschitz continuity in the more special Markovian case. The terminal condition will always be Lipschitz continuous in the process variables. To accommodate the fact that we do not require well-posedness of the problem we introduce the so called \emph{maximal interval}, which, roughly speaking, is the largest interval on which a given FBSDE system has reasonable solutions. The best case scenario is that the maximal interval coincides with $[0,T]$, such that the problem becomes well-posed. Based on our study of the form of the maximal interval and the behavior of the decoupling field at the left boundary, in the ill-posed case, we will propose a general method to verify well-posedness via contradiction. We have been able to apply this technique to construct solutions to various systems of strongly coupled multidimensional FBSDE appearing in problems like utility maximization in incomplete markets or the Skorohod embedding problem. This will be layed out in more detail in forthcoming work.

This paper is structured as follows. In Section \ref{prelim} we will define the notion of a decoupling field as done in \cite{jfzhang} and discuss some basic properties. Furthermore we will summarize some basic results about weak derivatives. Although some of theses statements might seem to be straightforward, we have included their proofs in the Appendix, since we have not been able to find a proper source to cite. \\
In Section \ref{local} we will prove a local existence and uniqueness result for decoupling fields (Theorem \ref{locexist}) for globally Lipschitz continuous coefficients. This result will serve as the fundament for the theory developed thereafter. The proof is constructive and is based on a contractive Picard-Lindel\"of iteration.\\
In Section \ref{examples} we discuss two simple examples to motivate the hypotheses of Theorem \ref{locexist}.\\
Section \ref{regularity} deals with regularity properties of decoupling fields. As a byproduct of the construction in Section \ref{local} we will obtain variational differentiability of solutions. More precisely, we show that $X,Y,Z$ depend in a weakly differentiable way on the initial vector $x$.\\
In Section \ref{global} we show global uniqueness and global regularity of decoupling fields, and study global existence by introducing the \emph{maximal interval}. We will prove a necessary condition for the problem to be ill-posed, and propose a general method to verify well-posedness in those cases in which it is conjectured. Our approach is related to the study of the \emph{characteristic BSDE} proposed in \cite{jfzhang}.\\
Furthermore we will discuss the Markovian case in more detail. Here our theory can be extended to coefficients that are not globally Lipschitz continuous, a very useful feature for applications. The decoupling field also assumes very nice properties in the Markovian case, such as being deterministic and continuous. We discuss the case in which the parameter functions are only locally Lipschitz in $Y,Z$. The case in which they are locally Lipschitz in $Z$ (and Lipschitz in the remaining components) deserves separate consideration. This will be left for future work.

\section{Preliminaries} \label{prelim}

\subsection{Decoupling Fields}

We will consider families $(\mu,\sigma,f)$ of measurable functions, more precisely
$$ \mu: [0,T]\times\Omega\times\mathbb{R}^n\times\mathbb{R}^m\times\mathbb{R}^{m\times d}\longrightarrow \mathbb{R}^n, $$
$$ \sigma: [0,T]\times\Omega\times\mathbb{R}^n\times\mathbb{R}^m\times\mathbb{R}^{m\times d}\longrightarrow \mathbb{R}^{n\times d}, $$
$$ f: [0,T]\times\Omega\times\mathbb{R}^n\times\mathbb{R}^m\times\mathbb{R}^{m\times d}\longrightarrow \mathbb{R}^m, $$
where
\begin{itemize}
\item $n,m,d\in\mathbb{N}$ and $T>0$,
\item $(\Omega,\mathcal{F},\mathbb{P},(\mathcal{F}_t)_{t\in[0,T]})$ is a complete filtered probability space,
\item $\mathcal{F}_0$ contains all null sets and also $\mathcal{F}_t=\sigma(\mathcal{F}_0,(W_s)_{s\in [0,t]})$ holds,
 where $(W_t)_{t\in[0,T]}$ is a $d$-dimensional Brownian motion, independent of $\mathcal{F}_0$,
\item $\mathcal{F}=\mathcal{F}_T$.
\end{itemize}
We want $\mu$, $\sigma$ and $f$ to be progressively measurable w.r.t. $(\mathcal{F}_t)_{t\in[0,T]}$, i.e. $\mu\mathbf{1}_{[0,t]},\sigma\mathbf{1}_{[0,t]},f\mathbf{1}_{[0,t]}$ must be $\mathcal{B}([0,T])\otimes\mathcal{F}_t\otimes\mathcal{B}(\mathbb{R}^n)\otimes\mathcal{B}(\mathbb{R}^m)\otimes\mathcal{B}(\mathbb{R}^{m\times d})$ - measurable for all $t\in[0,T]$. We will assume throughout the paper that $\mu$, $\sigma$ and $f$ have this property without mentioning it.

\begin{defi}
Let $\xi:\Omega\times\mathbb{R}^n\rightarrow\mathbb{R}^m$ be measurable and let $t\in[0,T]$.\\
We call a function $u:[t,T]\times\Omega\times\mathbb{R}^n\rightarrow\mathbb{R}^m$ with $u(T,\cdot)=\xi$ a.e. a \emph{decoupling field} for $\fbsde (\xi,(\mu,\sigma,f))$ on $[t,T]$ if for all $t_1,t_2\in[t,T]$ with $t_1\leq t_2$ and any $\mathcal{F}_{t_1}$ - measurable $X_{t_1}:\Omega\rightarrow\mathbb{R}^n$ there exist progressive processes $X,Y,Z$ on $[t_1,t_2]$ such that
\begin{itemize}
\item $X_s=X_{t_1}+\int_{t_1}^s\mu(r,X_r,Y_r,Z_r)\dx r+\int_{t_1}^s\sigma(r,X_r,Y_r,Z_r)\dx W_r$ a.s.,
\item $Y_s=Y_{t_2}-\int_{s}^{t_2}f(r,X_r,Y_r,Z_r)\dx r-\int_{s}^{t_2}Z_r\dx W_r$ a.s.,
\item $Y_s=u(s,X_s)$ a.s.,
\end{itemize}
for all $s\in[t_1,t_2]$. In particular we want all integrals to be well defined and $X,Y,Z$ to have values in $\mathbb{R}^n$, $\mathbb{R}^m$ and $\mathbb{R}^{m\times d}$ respectively.
\end{defi}

In the above definition the first equation is called the \emph{forward equation}, the second the \emph{backward equation} and the third will be referred to as the \emph{decoupling condition}. When we say that a triplet $(X,Y,Z)$ solves the FBSDE we mean, that it satisfies the forward and the backward equation, together with $Y_T=\xi(X_T)$ (we also require $t_2=T$).\\
Note that the \emph{terminal condition} $Y_T=\xi(X_T)$ is actually a consequence of the decoupling condition together with $u(T,\cdot)=\xi$.

At this point we do not require the triplet $(X,Y,Z)$ to be unique for given $t_1,t_2,X_{t_1}$.

Decoupling fields have the following very important property, which distinguishes them from classical solutions to FBSDEs.

\begin{lem}\label{glue}
If $u$ is a decoupling field for $\fbsde (\xi,(\mu,\sigma,f))$ on $[t,T]$ and a map
$\tilde{u}$ is a decoupling field for $\fbsde (u(t,\cdot),(\mu,\sigma,f))$ on $[s,t]$, where $0\leq s<t<T$, then the map $$\hat{u}:=\tilde{u}\mathbf{1}_{[s,t]}+u\mathbf{1}_{(t,T]}$$ is a decoupling field for $\fbsde  (\xi,(\mu,\sigma,f))$ on $[s,T]$.
\end{lem}
\begin{proof}
Assume we have a $t_1\in[s,t)$ and a $t_2\in(t,T]$. For any $\mathcal{F}_{t_1}$ - measurable $\hat{X}_{t_1}:\Omega\rightarrow\mathbb{R}^n$ we need to show existence of processes $\hat{X},\hat{Y},\hat{Z}$ solving our FBSDE on $[t_1,t_2]$ s.t. $\hat{Y}_r=\hat{u}(r,\hat{X}_r)$ a.s. for $r\in[t_1,t_2]$.

We construct theses processes in two steps: Firstly, we choose progressive processes $\tilde{X},\tilde{Y},\tilde{Z}$ on $[t_1,t]$ solving our FBSDE on $[t_1,t]$ with initial value $\hat{X}_{t_1}$ and satisfying $\tilde{Y}_r=\tilde{u}(r,\tilde{X}_r)=\hat{u}(r,\tilde{X}_r)$, $r\in[t_1,t]$, according to the definition of a decoupling field. \\
Moreover, there exist progressive processes $X,Y,Z$ on $[t,t_2]$ satisfying
\begin{itemize}
\item $X_r=\tilde{X}_t+\int_{t}^r\mu(v,X_v,Y_v,Z_v)\dx v+\int_{t}^r\sigma(v,X_v,Y_v,Z_v)\dx W_v$,
\item $Y_r=Y_{t_2}-\int_{r}^{t_2}f(v,X_v,Y_v,Z_v)\dx v-\int_{r}^{t_2}Z_v\dx W_v$,
\item $Y_r=u(r,X_r)=\hat{u}(r,X_r)$,
\end{itemize}
a.s. for all $r\in[t,t_2]$. Now define $\hat{X}$ on $[t_1,t_2]$ via $$\hat{X}:=\tilde{X}\mathbf{1}_{[t_1,t]}+X\mathbf{1}_{(t,t_2]}$$
and similarly define $\hat{Y}$ and $\hat{Z}$. \\
Note $\tilde{X}_t=X_t$ and also $\tilde{Y}_t=\tilde{u}(t,\tilde{X}_t)=u(t,X_t)=Y_t$. It is easy to check that $\hat{X},\hat{Y},\hat{Z}$ satisfy the FBSDE on $[t_1,t_2]$ and the decoupling condition.
\end{proof}

Note from the definition that if $u$ is a decoupling field and $\tilde{u}$ is a modification of $u$, i.e. for each $s\in[t,T]$ the functions $u(s,\omega,\cdot)$ and $\tilde{u}(s,\omega,\cdot)$ coincide for almost all $\omega\in\Omega$, then $\tilde{u}$ is also a decoupling field to the same problem. So $u$ could also be referred to as a class of modifications. Some of the representants of the class might be progressively measurable, others not. We will see below that a progressively measurable representant does exist if the decoupling field is Lipschitz continuous in $x$:

\begin{lem}
Let $u:[t,T]\times\Omega\times\mathbb{R}^n\rightarrow\mathbb{R}^m$ be a decoupling field to $\fbsde(\xi,(\mu,\sigma,f))$ which is Lipschitz continuous in $x\in\mathbb{R}^n$ in the sense that there exists a constant $L>0$ s.t. for every $s\in[t,T]$:
$$|u(s,\omega,x)-u(s,\omega,x')|\leq L|x-x'|\qquad\forall x,x'\in\mathbb{R}^n, \qquad\textrm{ for a.a. }\omega\in\Omega.$$
Then $u$ has a modification $\tilde{u}$ which is
\begin{itemize}
\item progressively measurable,
\item Lipschitz continuous in $x$ in the strong sense $$|\tilde{u}(s,\omega,x)-\tilde{u}(s,\omega,x')|\leq L|x-x'|\qquad \forall s\in[t,T],\,\omega\in\Omega,\, x,x'\in\mathbb{R}^n$$
\item and "weakly right-continuous" in the sense that
$$\lim_{n\rightarrow\infty}\tilde{u}(s_n,\cdot,x_n)=\tilde{u}(s',\cdot,x')\qquad\textrm{a.s.},$$
for all $(s',x')\in[t,T]\times\mathbb{R}^n$ and all sequences $(s_n)\subset [s',T]$, $(x_n)\subset \mathbb{R}^n$, converging to $s'$ and $x'$ respectively.
\end{itemize}
\end{lem}
\begin{proof}
We can assume without loss of generality that $u$ is truly Lipschitz continuous in $x$ with Lipschitz constant $L$ by modifying it for every fixed $s\in[t,T]$ such that $u(s,\omega,\cdot)$ is set to $0$, if it is not Lipschitz continuous with constant $L$. This will have to be done for a set of $\omega$, which has measure zero (for each fixed $s\in[t,T]$).

Choose any $t_1$ from the interval $[t,T]$ and some $x'\in\mathbb{R}^n$ as initial value of the problem
$$ X_s=x'+\int_{t_1}^s\mu(r,X_r,Y_r,Z_r)\dx r+\int_{t_1}^s\sigma(r,X_r,Y_r,Z_r)\dx W_r, $$
$$ Y_s=Y_T-\int_{s}^{T}f(r,X_r,Y_r,Z_r)\dx r-\int_{s}^{t_2}Z_r\dx W_r, $$
$$ Y_s=u(s,X_s). $$
Let $x$ be another point in $\mathbb{R}^n$ and choose any $t_2\in [t_1,T]$. We can assume that $X$ and $Y$ are continuous (we can choose such modifications). We use the triangle inequality together with the decoupling condition $Y_s=u(s,X_s)$:
\begin{multline*}
|u(t_2,x)-u(t_1,x')|\leq |u(t_2,x)-u(t_2,x')|+|u(t_2,x')-u(t_2,X_{t_2})|+|u(t_2,X_{t_2})-u(t_1,x')|\leq \\
=L|x-x'|+|u(t_2,X_{t_1})-u(t_2,X_{t_2})|+|Y_{t_2}-Y_{t_1}|\leq L|x-x'|+L|X_{t_2}-X_{t_1}|+|Y_{t_2}-Y_{t_1}|.
\end{multline*}
Choosing sequences $t^{(n)}_2\downarrow t_1$ and $x^{(n)}\rightarrow x'$, we obtain
$$\lim_{n\rightarrow\infty}u(t^{(n)}_2,\cdot,x^{(n)})=u(t_1,\cdot,x') \qquad\textrm{a.s. }$$
from the continuity of the processes $X$ and $Y$. Now define $\tilde{u}$ via
$$ \tilde{u}(s,\omega,x):=\limsup_{n\rightarrow\infty}\sum_{m=1}^n u\left(t+m\frac{T-t}{n},\omega,x\right)\mathbf{1}_{\left(t+(m-1)\frac{T-t}{n},t+m\frac{T-t}{n}\right]}(s).$$
Now observe:
\begin{itemize}
\item $\tilde{u}$ clearly inherits the (strong) Lipschitz continuity in $x$ from $u$.
\item For any $s\in[t,T]$ and any $\varepsilon>0$ the function $\tilde{u}\mathbf{1}_{[t,s]}$ is $\mathcal{B}([0,T])\otimes\mathcal{F}_{s+\varepsilon}\otimes\mathcal{B}(\mathbb{R}^n)$ - measurable, since  $u\left(t+m\frac{T-t}{n},\omega,x\right)\mathbf{1}_{\left(t+(m-1)\frac{T-t}{n},t+m\frac{T-t}{n}\right]\cap[t,s]}$ is measurable w.r.t. this $\sigma$-algebra if $n$ is large enough. Thus $\tilde{u}\mathbf{1}_{[t,s]}$ is $\mathcal{B}([0,T])\otimes\mathcal{F}_{s+}\otimes\mathcal{B}(\mathbb{R}^n)$ - measurable for all $s$ and so $\tilde{u}$ is progressively measurable due to $\mathcal{F}_{s+}=\mathcal{F}_{s}$.
\item For all $s\in[t,T]$ and all $x\in\mathbb{R}^n$ the random variables $\tilde{u}(s,\cdot,x)$ and $u(s,\cdot,x)$ are a.s. equal, since $\lim_{n\rightarrow\infty}u\left(t+m(s,n)\frac{T-t}{n},\omega,x\right)=u(s,\omega,x)$ for a.a. $\omega$, where $m(s,n)$ is the unique element of $\{1,\ldots,n\}$ s.t. $\mathbf{1}_{\left(t+(m(s,n)-1)\frac{T-t}{n},t+m(s,n)\frac{T-t}{n}\right]}(s)=1$. Note here that $t+m(s,n)\frac{T-t}{n}\geq s$ converges to $s$ for $n\rightarrow\infty$. \\ Due to Lipschitz continuity in $x$ the maps $x\mapsto \tilde{u}(s,\omega,x)$ and $x\mapsto u(s,\omega,x)$ must also coincide for a.a. $\omega$ and so $\tilde{u}$ is indeed a modification of $u$.
\item As a modification $\tilde{u}$ inherits the "weak right-continuity" of $u$.
\end{itemize}
\end{proof}

\subsection{Weak Derivatives}

In this paper we will work extensively with weak derivatives. This will allow us to show variational differentiability (i.e. w.r.t. the initial value $x\in\mathbb{R}^n$) of the processes $X,Y,Z$ for Lipschitz continuous $\mu,\sigma,f,\xi$. \\
We start by fixing notation and giving some definitions.

For the following $|\cdot|$ will denote the usual square norm in any finite dimensional Euclidean space.\\
We can interpret elements of $\mathbb{R}^{n\times d}$ and $\mathbb{R}^{m\times d}$ as matrices or as linear operators from $\mathbb{R}^{d}$ with values in $\mathbb{R}^{n}$ or $\mathbb{R}^{m}$. Similarly we interpret $\mathbb{R}^{m \times d \times n}$ as the space of linear mappings from $\mathbb{R}^n$ to $\mathbb{R}^{m \times d}$.\\
If $x\in\mathbb{R}^{m\times d}$ or $x\in\mathbb{R}^{n\times d}$ the expression $|x|$ denotes the Frobenius norm of the linear operator $x$, i.e. the square root of the sum of the squares of its matrix coefficients. \\
If $x\in\mathbb{R}^{n\times n}$ or $x\in\mathbb{R}^{m\times n}$ or $x\in\mathbb{R}^{m \times d \times n}$ or $x\in\mathbb{R}^{n \times d \times n}$ we define $|x|_v:=|x\cdot v|$ for all $v\in S^{n-1}$, where $\cdot$ is the application of the linear operator $x$ to the vector $v$ such that $x\cdot v$ is in $\mathbb{R}^{n}$ or $\mathbb{R}^{m}$ or $\mathbb{R}^{m \times d}$ or $\mathbb{R}^{n \times d}$.

We denote by $L_{\xi,x}$ the Lipschitz constant of a map $\xi: \Omega\times\mathbb{R}^n\rightarrow \mathbb{R}^m$ w.r.t. the Euclidean norms, i.e.
$$ L_{\xi,x}:=\inf\left\{L\geq 0\,|\,|\xi(\omega,x_1)-\xi(\omega,x_2)|\leq L|x_1-x_2|\forall x_1,x_2\textrm{ for a.a. $\omega$}\right\}, $$
where $\inf \emptyset:=\infty$. Note that $\xi$ is Lipschitz continuous w.r.t. $x\in\mathbb{R}^n$ if and only if $L_{\xi,x}<\infty$. \\

Consider a mapping $X:\mathcal{M}\times\Lambda\rightarrow\mathbb{R}$, where $(\mathcal{M},\mathcal{A},\rho)$ is some complete measure space and $\Lambda\subseteq\mathbb{R}^N$ is open, $N\in\mathbb{N}$. We say that $X$ is \emph{weakly differentiable w.r.t.} the parameter $\lambda\in\Lambda$, if for almost all $\omega\in\mathcal{M}$ the mapping $X(\omega,\cdot):\Lambda\rightarrow\mathbb{R}$ is weakly differentiable. This means that there exists a mapping $\frac{\dx}{\dx\lambda}X:\mathcal{M}\times\Lambda\rightarrow\mathbb{R}^{1\times N}$ such that
$$ \int_{\Lambda}\varphi(\lambda)\frac{\dx}{\dx\lambda}X(\omega,\lambda)\dx\lambda=
-\int_{\Lambda}X(\omega,\lambda)\frac{\dx}{\dx\lambda}\varphi(\lambda)\dx\lambda,$$
for any real valued test function $\varphi\in C^{\infty}_c(\Lambda)$ and almost all $\omega\in\mathcal{M}$. In particular $X(\omega,\cdot)$ and $\frac{\dx}{\dx\lambda}X(\omega,\cdot)$ have to be locally integrable for a.a. $\omega$. This of course includes measurability for almost every \emph{fixed} $\omega$. \\
Similarly we could also define higher order weak differentiability. Weak differentiability for vector valued mappings is defined componentwise. \\
We call two maps $Y,Z:\mathcal{M}\times\Lambda\rightarrow\mathbb{R}^{1\times N}$ modifications of each other if $Y(\omega,\cdot)$ and $Z(\omega,\cdot)$ are a.e. equal for almost every fixed $\omega$. Obviously a modification of a weak derivative is again a weak derivative (of the same $X$). \\
If $X$ is a \emph{measurable} a function of $(\omega,\lambda)$, its weak derivative $\frac{\dx}{\dx\lambda}X$ will have a measurable modification: For all $v\in \mathbb{R}^N$ and all $h>0$ we can write
\begin{equation}\label{fundu}\int_0^h\frac{\dx}{\dx\lambda}X(\omega,\lambda_0+t v)v\dx t=X(\omega,\lambda_0+hv)-X(\omega,\lambda_0),\end{equation}
for a.a. $\lambda_0\in\Lambda$, s.t. $\overline{B_{h|v|}(\lambda_0)}\subseteq\Lambda$, for almost every $\omega\in\mathcal{M}$ (Lemma \ref{fundulem}). For instance choose $h=h_n=n^{-1}$, $n\in\mathbb{N}$. Clearly $Y(\omega,\lambda_0):=\limsup_{n\rightarrow\infty}\frac{1}{h_n}(X(\omega,\lambda_0+h_nv)-X(\omega,\lambda_0))$ is a measurable function of $(\omega,\lambda_0)$. However $Y$ is a modification of $\frac{\dx}{\dx\lambda}Xv$ due to (\ref{fundu}) and Lebesgue's differentiation theorem. This allows us to construct a measurable modification of $\frac{\dx}{\dx\lambda}X$ by taking canonical unit vectors for $v$. \\
The relationship $\frac{\dx}{\dx\lambda}X(\omega,\lambda_0)v=\limsup_{n\rightarrow\infty}\frac{1}{h_n}(X(\omega,\lambda_0+h_nv)-X(\omega,\lambda_0))$, which holds for almost all $\lambda_0$, for almost all $\omega$, also implies uniqueness of $\frac{\dx}{\dx\lambda}X$ up to modifications. \\

If a map $\xi: \Omega\times\mathbb{R}^n\rightarrow \mathbb{R}^m$ is measurable and $L_{\xi,x}<\infty$ then $\xi$ is weakly differentiable w.r.t $x$ (according to Rademacher's theorem) and we can also write
$$ L_{\xi,x}=\esssup\left\{\left|\frac{\dx}{\dx x} \xi(\omega,x)\right|_v\,\Bigg|\,\omega\in\Omega,x\in\mathbb{R}^n,v\in S^{n-1}\right\},$$
where the weak derivative $\frac{\dx}{\dx x} \xi$ has values in $\mathbb{R}^{m\times n}$, so $\frac{\dx}{\dx x} \xi(\omega,x)v\in\mathbb{R}^m$, if $v\in\mathbb{R}^n$. \\

Note that we have the following "chain rule" for weak derivatives.


\begin{lem}\label{chainrule}
Let $g: \mathcal{M}\times\mathbb{R}^d\rightarrow \mathbb{R}^m$ be measurable s.t. $L_{g,x}<\infty$. Furthermore let $X_i: \mathcal{M}\times\mathbb{R}^n\rightarrow \mathbb{R}^{d_i}$, $i=1,\ldots,k$ be measurable and weakly differentiable w.r.t. $\lambda\in\mathbb{R}^n$. Let $X:=(X_1,\ldots,X_k)$ be $\mathbb{R}^{d}$-valued, i.e. assume $\sum_{i=1}^k d_i=d$. \\
Then the measurable mapping $g(X): \mathcal{M}\times\mathbb{R}^n\rightarrow \mathbb{R}^{m}$ is also weakly differentiable w.r.t. $\lambda\in\mathbb{R}^n$ and furthermore there exist measurable mappings $\Delta^{X}_{x_i}g:\mathcal{M}\times\mathbb{R}^n\times\mathbb{R}^n\rightarrow\mathbb{R}^{m\times d_i} $ s.t.
\begin{itemize}
\item $\sup_{v\in S^{d_i-1}}|\Delta^{X}_{x_i}g|_v\leq L_{g,x_i}$ everywhere for every $i=1,\ldots,k$,
\item for all $v\in \mathbb{R}^n$
$$ \left(\frac{\dx}{\dx \lambda}g(X)(\omega,\lambda)\right)v=\sum_{i=1}^k\left(\Delta^{X}_{x_i}g(\omega,\lambda,v)\right)\left(\frac{\dx}{\dx \lambda}X_i(\omega,\lambda)\right)v$$
holds for almost all $\lambda\in\mathbb{R}^d$, $\omega\in\mathcal{M}$.
\end{itemize}
\end{lem}
\begin{proof}
See Appendix.
\end{proof}

For later reference we state Lemmas \ref{wd1}-\ref{wd5}. They will be needed to justify interchanging differentiation (in the weak sense) with integration w.r.t. time, probability measure or Brownian motion.

\begin{lem}\label{wd1}
Let $X:\Omega\times\mathbb{R}^n\rightarrow\mathbb{R}$ be measurable, s.t.
\begin{itemize}
\item $X$ is weakly differentiable w.r.t. $\lambda$,
\item $\mathbb{E}\left[\left|X(\cdot,\lambda)\right|\right]$ and $\mathbb{E}\left[\left|\frac{\dx}{\dx\lambda}X(\cdot,\lambda)\right|\right]$ are both locally integrable w.r.t. $\lambda$.
\end{itemize}
Let also $\mathcal{G}\subseteq\mathcal{F}$ be a $\sigma$-algebra. Then the mapping $(\omega,\lambda)\mapsto\mathbb{E}[X(\cdot,\lambda)|\mathcal{G}](\omega)$ is measurable and weakly differentiable w.r.t. $\lambda$ and
$\frac{\dx}{\dx\lambda}\mathbb{E}[X(\cdot,\lambda)|\mathcal{G}]=\mathbb{E}\left[\frac{\dx}{\dx\lambda}X(\cdot,\lambda)|\mathcal{G}\right]$.
\end{lem}
\begin{proof}
See Appendix.
\end{proof}

\begin{lem}\label{wd2}
Let $Z:[0,T]\times\Omega\times\mathbb{R}^n\rightarrow\mathbb{R}$ be measurable, s.t.
\begin{itemize}
\item $Z$ is weakly differentiable w.r.t. $\lambda\in\mathbb{R}^n$,
\item $\mathbb{E}\left[\int_0^T\left|Z_s(\cdot,\lambda)\right|\dx s\right]$ and $\mathbb{E}\left[\int_0^T\left|\frac{\dx}{\dx\lambda}Z_s(\cdot,\lambda)\right|\dx s\right]$ are both locally integrable w.r.t. $\lambda$.
\end{itemize}
Then the mapping $X:=\int_0^TZ_s\dx s:\Omega\times\mathbb{R}^n\rightarrow\mathbb{R}$ is measurable, weakly differentiable w.r.t. $\lambda\in\mathbb{R}^n$ and
$\frac{\dx}{\dx\lambda}X(\cdot,\lambda)=\int_0^T\frac{\dx}{\dx\lambda}Z_s(\cdot,\lambda)\dx s$.
\end{lem}
\begin{proof}
Define a new probability space $\left([0,T]\times \Omega,\mathcal{L}([0,T])\otimes\mathcal{F},\frac{1}{T} dt|_{\mathcal{B}[0,T]}\otimes\mathbb{P}\right)$, define $\mathcal{G}:=\{\emptyset,[0,T]\}\otimes\mathcal{F}$ and apply Lemma \ref{wd1}. Here $\mathcal{L}([0,T])$ is the $\sigma$-Algebra of Lebesgue measurable subsets of $[0,T]$.
\end{proof}


\begin{lem}\label{wd3}
Let $Z:[0,T]\times\Omega\times\mathbb{R}^n\rightarrow\mathbb{R}^d$ be progressively measurable, s.t.
\begin{itemize}
\item $Z$ is weakly differentiable w.r.t. $\lambda\in\mathbb{R}^n$,
\item $\mathbb{E}\left[\int_0^T\left|Z_s(\cdot,\lambda)\right|^2\dx s\right]$ and $\mathbb{E}\left[\int_0^T\left|\frac{\dx}{\dx\lambda}Z_s(\cdot,\lambda)\right|^2\dx s\right]$ are both locally integrable w.r.t. $\lambda$.
\end{itemize}
Then the mapping $X:=\int_0^TZ_s^\top\dx W_s:\Omega\times\mathbb{R}^n\rightarrow\mathbb{R}$ is measurable, weakly differentiable w.r.t. $\lambda\in\mathbb{R}^n$ and
$\frac{\dx}{\dx\lambda}X(\cdot,\lambda)=\int_0^T\frac{\dx}{\dx\lambda}Z_s(\cdot,\lambda)^\top\dx W_s$.
\end{lem}
\begin{proof}
See Appendix.
\end{proof}

Conversely, we can also show:

\begin{lem}\label{wd4}
Let $X:\Omega\times\mathbb{R}^n\rightarrow\mathbb{R}$ be measurable, s.t.
\begin{itemize}
\item $X$ is weakly differentiable w.r.t. $\lambda$,
\item $\mathbb{E}\left[\left|X(\cdot,\lambda)\right|^2\right]$ and $\mathbb{E}\left[\left|\frac{\dx}{\dx\lambda}X(\cdot,\lambda)\right|^2\right]$ are both locally integrable w.r.t. $\lambda$.
\end{itemize}
Then the unique progressively measurable process $Z:\Omega\times[0,T]\times\mathbb{R}^n\rightarrow\mathbb{R}^d$ s.t. $X=\mathbb{E}[X]+\int_0^TZ_s^\top\dx W_s$ is weakly differentiable w.r.t. $\lambda$ and
$\frac{\dx}{\dx\lambda}X(\cdot,\lambda)=\mathbb{E}\left[\frac{\dx}{\dx\lambda}X(\cdot,\lambda)\right]+\int_0^T\frac{\dx}{\dx\lambda}Z_s(\cdot,\lambda)^\top\dx W_s$.
\end{lem}
\begin{proof}
See Appendix.
\end{proof}

\begin{lem}\label{wd5}
Let $X:\Omega\times\mathbb{R}^n\rightarrow\mathbb{R}$ be measurable and $V:\Omega\times[0,T]\times\mathbb{R}^n\rightarrow\mathbb{R}$ be progressively measurable s.t.
\begin{itemize}
\item $X$ and $V$ are weakly differentiable w.r.t. $\lambda\in\mathbb{R}^n$,
\item $\mathbb{E}\left[\left|X(\cdot,\lambda)\right|^2\right]$ and $\mathbb{E}\left[\left|\frac{\dx}{\dx\lambda}X(\cdot,\lambda)\right|^2\right]$ are both locally integrable w.r.t. $\lambda$,
\item $\mathbb{E}\left[\left(\int_0^T\left|V_s(\cdot,\lambda)\right|\dx s\right)^2\right]$ and $\mathbb{E}\left[\left(\int_0^T\left|\frac{\dx}{\dx\lambda}V_s(\cdot,\lambda)\right|\dx s\right)^2\right]$ are both locally integrable w.r.t. $\lambda$.
\end{itemize}
Then there exist unique progressive processes $Y,Z:\Omega\times[0,T]\times\mathbb{R}^n\rightarrow\mathbb{R},\mathbb{R}^d$ s.t.
$$ Y_t=X-\int_t^TV_s\dx s-\int_t^TZ_s^\top\dx W_s, $$
$Y$ and $Z$ are both weakly differentiable w.r.t. $\lambda$ and
$$ \frac{\dx}{\dx\lambda}Y_t=\frac{\dx}{\dx\lambda}X-\int_t^T\frac{\dx}{\dx\lambda}V_s\dx s-\int_t^T\frac{\dx}{\dx\lambda}Z_s^\top\dx W_s. $$
\end{lem}
\begin{proof}
See Appendix.
\end{proof}

Finally, we will need the following stability result.

\begin{lem}\label{weakderivexist}
Let $(\mathcal{M},\mathcal{A},\rho)$ be some finite and complete measure space and let $\Lambda\subseteq\mathbb{R}^N$ be open. Let $(X_i)_{i\in\mathbb{N}}$ be a sequence of measurable real valued maps on $\Lambda\times\mathcal{M}$ s.t. $X_i(\cdot,\omega)$ has all weak derivatives up to order $\delta\in\mathbb{N}$ for almost all $\omega\in\mathcal{M}$ and s.t. there exists a constant $C<\infty$ with
$$\sum_{1\leq |\alpha|\leq\delta}\int_{\mathcal{M}}|D_\lambda^\alpha X_i(\lambda,\cdot)|^2\dx\rho\leq C, $$
for almost all $\lambda\in\Lambda$ and all $i\in\mathbb{N}$, where $\alpha\in\mathbb{N}^N$ is a multi-index.\\
Assume further that there exists a real valued map $X$ on $\Lambda\times\mathcal{M}$ such that $\lim_{i\rightarrow\infty}X_i(\lambda,\cdot)=X(\lambda,\cdot)$ in $\mathcal{L}^2$ for almost all $\lambda\in\Lambda$. \\
Then $X$ is measurable and $X(\cdot,\omega)$ has all weak derivatives up to order $\delta$ for almost all $\omega\in\mathcal{M}$ and satisfies $\sum_{1\leq|\alpha|\leq\delta}\int_{\mathcal{M}}|D_\lambda^\alpha X(\lambda,\cdot)|^2\dx\rho\leq C$ for almost all $\lambda\in\Lambda$.
\end{lem}
\begin{proof}
See Appendix.
\end{proof}

\section{Local Existence and Uniqueness} \label{local}

We denote by $L_{\sigma,z}$ the Lipschitz constant of $\sigma$ w.r.t. the dependence on the last component $z$ (and w.r.t. the Frobenius norms on $\mathbb{R}^{m\times d}$ and $\mathbb{R}^{n\times d}$).\\
By $L_{\sigma,z}^{-1}=\frac{1}{L_{\sigma,z}}$ we mean $\frac{1}{L_{\sigma,z}}$ if $L_{\sigma,z}>0$ and $\infty$ otherwise.

In the following we need further notation. For an integrable real valued random variable $X$ the expression $\mathbb{E}_t[X]$ refers to $\mathbb{E}[X|\mathcal{F}_t]$, while $\mathbb{E}_{\hat{t},\infty}[X]$ refers to $\esssup\,\mathbb{E}[X|\mathcal{F}_t]$, which might be $\infty$ or even $-\infty$, but is always well defined as the infimum of all constants $c\in[-\infty,\infty]$ such that $\mathbb{E}[X|\mathcal{F}_t]\leq c$ a.s.. \\
As usual $\|X\|_\infty$ refers to the essential supremum of $|X|$.

\begin{thm}\label{locexist}
Let
\begin{itemize}
\item $\mu,\sigma,f$ be Lipschitz continuous in $(x,y,z)$ with Lipschitz constant $L$ s.t.
\item $\left\|\left(|\mu|+|f|+|\sigma|\right)(\cdot,\cdot,0,0,0)\right\|_{\infty}<\infty$.
\item $\xi: \Omega\times\mathbb{R}^n\rightarrow \mathbb{R}^m$ be measurable s.t. $\|\xi(\cdot,0)\|_{\infty}<\infty$ and $L_{\xi,x}<L_{\sigma,z}^{-1}$.
\end{itemize}
Then there exists a time $t\in[0,T)$ such that $\fbsde (\xi,(\mu,\sigma,f))$ has a unique (up to modification) decoupling field $u$ on $[t,T]$ with $L_{u,x}<L_{\sigma,z}^{-1}$ and $\|u(\cdot,\cdot,0)\|_{\infty}<\infty$.
\end{thm}
\begin{proof}
Let for some $t\in[0,T)$, which will be specified later, $X_t: \mathbb{R}^n\times\Omega\longrightarrow\mathbb{R}^{n}$ be a $\mathcal{B}(\mathbb{R}^n)\otimes\mathcal{F}_t$ - measurable function s.t. $X_t(\cdot,\omega)$ is weakly differentiable for almost all $\omega\in\Omega$ and
$$\esssup_{\lambda\in\mathbb{R}^n}\sup_{v\in S^{n-1}}\mathbb{E}_{\hat{t},\infty}\left[\left|\frac{\dx}{\dx \lambda}X_t(\lambda,\cdot)\right|^2_v\right]<\infty,$$
for some $\hat{t}\in [0,t]$. \\
Assume furthermore that $\mathbb{E}_{\hat{t},\infty}\left[|X_t(\lambda,\cdot)|^2\right]<\infty$ for all $\lambda\in\mathbb{R}^{n}$. \\
We want to solve the coupled FBSDE
\begin{itemize}
\item $X_s=X_t+\int_{t}^s\mu(r,X_r,Y_r,Z_r)\dx r+\int_{t}^s\sigma(r,X_r,Y_r,Z_r)\dx W_r,$
\item $Y_s=\xi(X_T)-\int_{s}^{T}f(r,X_r,Y_r,Z_r)\dx r-\int_{s}^{T}Z_r\dx W_r,$
\end{itemize}
which means that $X,Y,Z$ would be functions of $\lambda,\omega$ and $s$ and the two equations would hold for almost all $(\lambda,\omega,s)\in\mathbb{R}^{n}\times\Omega\times [t,T]$. \\
Let $\mathbb{G}_{\hat{t}}$ be the space of all progressive $\mathbb{R}^{n}\times\mathbb{R}^{m}\times\mathbb{R}^{n\times d}$ - valued processes $(X,Y,Z)$ on $[t,T]\times\Omega$ s.t.
\begin{multline*}
\|(X,Y,Z)\|_w:=\max\Bigg(\sup_{s\in[t,T]}\sqrt{\mathbb{E}_{\hat{t},\infty}[|X_s|^2]},\, (1+L_{\sigma,z})\sup_{s\in[t,T]}\sqrt{\mathbb{E}_{\hat{t},\infty}[|Y_s|^2]},\, \\
(1+L_{\sigma,z})\sqrt{\mathbb{E}_{\hat{t},\infty}\left[\int_t^T|Z_s|^2\dx s\right]}\Bigg)<\infty.
\end{multline*}
This means, that if $(X,Y,Z)$ also depends on a parameter $\lambda$, then $\|(X,Y,Z)\|_w$ would depend on $\lambda$ as well. \\
Let $\mathbb{H}$ be the space of all progressive mappings
$$ (X,Y,Z): \mathbb{R}^{n}\times[t,T]\times\Omega\longrightarrow\mathbb{R}^{n\times n}\times\mathbb{R}^{m\times n}\times\mathbb{R}^{n\times d\times n}$$
such that
$$\|(X,Y,Z)\|_s:=\esssup_{\lambda\in\mathbb{R}^{n}}\sup_{v\in S^{n-1}}\,\left\|(X(\lambda,\cdot)v,Y(\lambda,\cdot)v,Z(\lambda,\cdot)v)\right\|_w
<\infty.$$

Now fix $\lambda\in\mathbb{R}^{n}$! \\
For any $(X^0,Y^0,Z^0)\in\mathbb{G}_{\hat{t}}$
there are unique $(X^1,Y^1,Z^1)=F(X^0,Y^0,Z^0)\in\mathbb{G}_{\hat{t}}$ s.t.
$$ X^{1}_s:=X_t+\int_{t}^s\mu(r,X^{0}_r,Y^{0}_r,Z^{0}_r)\dx r+\int_{t}^s\sigma(r,X^{0}_r,Y^{0}_r,Z^{0}_r)\dx W_r. $$
$$ Y^{1}_s:=\xi(X^{1}_T)-\int_{s}^{T}f(r,X^{1}_r,Y^{0}_r,Z^{0}_r)\dx r-\int_{s}^{T}(Z^{1}_r)\dx W_r, $$
for almost all $s,\lambda,\omega$. We assume that this is clear. This defines the mapping $F:\mathbb{G}_{\hat{t}}\rightarrow \mathbb{G}_{\hat{t}}$. \\
In the sequel we will check that $F$ is a contraction w.r.t. $\|\cdot\|_w$ if $t$ is close enough to $T$, depending on the Lipschitz constant $L$ for $(\mu,\sigma,f,g)$, $L_{\sigma,z}$ and $L_{\xi,x}$.\\
Let $(X^0,Y^0,Z^0),(\tilde{X}^0,\tilde{Y}^0,\tilde{Z}^0)\in\mathbb{G}_{\hat{t}}$ and accordingly $(X^1,Y^1,Z^1)=F(X^0,Y^0,Z^0),(\tilde{X}^1,\tilde{Y}^1,\tilde{Z}^1)=F(\tilde{X}^0,\tilde{Y}^0,\tilde{Z}^0)\in\mathbb{G}_{\hat{t}}$.
We obviously have
$$ X^{1}_s-\tilde{X}^{1}_s=\int_{t}^s\mu(r,X^{0}_r,Y^{0}_r,Z^{0}_r)-\mu(r,\tilde{X}^0_r,\tilde{Y}^0_r,\tilde{Z}^0_r)\dx r+
\int_{t}^s\sigma(r,X^{0}_r,Y^{0}_r,Z^{0}_r)-\sigma(r,\tilde{X}^0_r,\tilde{Y}^0_r,\tilde{Z}^0_r)\dx W_r $$
and therefore
\begin{multline*}
\left(\mathbb{E}_{\hat{t}}\left[\left|X^{1}_s-\tilde{X}^{1}_s\right|^2\right]\right)^{\frac{1}{2}}\leq
\left(\mathbb{E}_{\hat{t}}\left[\left|\int_{t}^s\mu(r,X^{0}_r,Y^{0}_r,Z^{0}_r)-\mu(r,\tilde{X}^0_r,\tilde{Y}^0_r,\tilde{Z}^0_r)\dx r\right|^2\right]\right)^{\frac{1}{2}}+ \\
	 +\left(\mathbb{E}_{\hat{t}}\left[\left|\int_{t}^s\sigma(r,X^{0}_r,Y^{0}_r,Z^{0}_r)-\sigma(r,\tilde{X}^0_r,\tilde{Y}^0_r,\tilde{Z}^0_r)\dx W_r\right|^2\right]\right)^{\frac{1}{2}}\leq
\end{multline*}
\begin{multline*}
	\leq L\left(\mathbb{E}_{\hat{t}}\left[\left(\int_{t}^s|X^{0}_r-\tilde{X}^0_r|+|Y^{0}_r-\tilde{Y}^0_r|+|Z^{0}_r-\tilde{Z}^0_r|\dx r\right)^2\right]\right)^{\frac{1}{2}}+ \\
	 +\left(\mathbb{E}_{\hat{t}}\left[\int_{t}^s\left(L|X^{0}_r-\tilde{X}^0_r|+L|Y^{0}_r-\tilde{Y}^0_r|+L_{\sigma,z}|Z^{0}_r-\tilde{Z}^0_r|\right)^2\dx r\right]\right)^{\frac{1}{2}}\leq
\end{multline*}
\begin{multline*}
	\leq L\left(\mathbb{E}_{\hat{t}}\left[\left(\int_{t}^s|X^{0}_r-\tilde{X}^0_r|\dx r\right)^2\right]\right)^{\frac{1}{2}}+
L\left(\mathbb{E}_{\hat{t}}\left[\left(\int_{t}^s|Y^{0}_r-\tilde{Y}^0_r|\dx r\right)^2\right]\right)^{\frac{1}{2}}+ \\
	\shoveright{+L\left(\mathbb{E}_{\hat{t}}\left[\left(\int_{t}^s|Z^{0}_r-\tilde{Z}^0_r|\dx r\right)^2\right]\right)^{\frac{1}{2}}+ }\\
	\shoveleft{+L\left(\mathbb{E}_{\hat{t}}\left[\int_{t}^s|X^{0}_r-\tilde{X}^0_r|^2\dx r\right]\right)^{\frac{1}{2}}+
L\left(\mathbb{E}_{\hat{t}}\left[\int_{t}^s|Y^{0}_r-\tilde{Y}^0_r|^2\dx r\right]\right)^{\frac{1}{2}}+} \\
	+L_{\sigma,z}\left(\mathbb{E}_{\hat{t}}\left[\int_{t}^s|Z^{0}_r-\tilde{Z}^0_r|^2\dx r\right]\right)^{\frac{1}{2}}\leq
\end{multline*}
\begin{multline*}
	\leq L\sqrt{s-t}\left(\mathbb{E}_{\hat{t}}\left[\int_{t}^s|X^{0}_r-\tilde{X}^0_r|^2\dx r\right]\right)^{\frac{1}{2}}+
L\sqrt{s-t}\left(\mathbb{E}_{\hat{t}}\left[\int_{t}^s|Y^{0}_r-\tilde{Y}^0_r|^2\dx r\right]\right)^{\frac{1}{2}}+ \\
	\shoveright{+L\sqrt{s-t}\left(\mathbb{E}_{\hat{t}}\left[\int_{t}^s|Z^{0}_r-\tilde{Z}^0_r|^2\dx r\right]\right)^{\frac{1}{2}}+} \\
	 \shoveleft{+L\sqrt{s-t}\left(\sup_{r\in[t,T]}\mathbb{E}_{\hat{t},\infty}\left[|X^{0}_r-\tilde{X}^0_r|^2\right]\right)^{\frac{1}{2}}+
L\sqrt{s-t}\left(\sup_{r\in[t,T]}\mathbb{E}_{\hat{t},\infty}\left[|Y^{0}_r-\tilde{Y}^0_r|^2\right]\right)^{\frac{1}{2}}+} \\
	+L_{\sigma,z}\left(\mathbb{E}_{\hat{t},\infty}\left[\int_{t}^s|Z^{0}_r-\tilde{Z}^0_r|^2\dx r\right]\right)^{\frac{1}{2}}\leq
\end{multline*}
\begin{multline*}
	\leq L\left(T-t+\sqrt{T-t}\right)\left(\sup_{r\in[t,T]}\sqrt{\mathbb{E}_{\hat{t},\infty}\left[|X^{0}_r-\tilde{X}^0_r|^2\right]}+\sup_{r\in[t,T]}\sqrt{\mathbb{E}_{\hat{t},\infty}\left[|Y^{0}_r-\tilde{Y}^0_r|^2\right]}\right)+ \\
	 +\left(L_{\sigma,z}+L\sqrt{T-t}\right)\left(\mathbb{E}_{\hat{t},\infty}\left[\int_{t}^T|Z^{0}_r-\tilde{Z}^0_r|^2\dx r\right]\right)^{\frac{1}{2}}\leq
\end{multline*}
\begin{multline*}
	\leq L\left(T-t+\sqrt{T-t}\right)\sup_{r\in[t,T]}\sqrt{\mathbb{E}_{\hat{t},\infty}\left[|X^{0}_r-\tilde{X}^0_r|^2\right]}+ \\
	 +L\left(T-t+\sqrt{T-t}\right)(1+L_{\sigma,z})\sup_{r\in[t,T]}\sqrt{\mathbb{E}_{\hat{t},\infty}\left[|Y^{0}_r-\tilde{Y}^0_r|^2\right]}+ \\
	 \shoveright{+\frac{L_{\sigma,z}+L\sqrt{T-t}}{1+L_{\sigma,z}}(1+L_{\sigma,z})\left(\mathbb{E}_{\hat{t},\infty}\left[\int_{t}^T|Z^{0}_r-\tilde{Z}^0_r|^2\dx r\right]\right)^{\frac{1}{2}}\leq }\\
	\leq \left(2\cdot L\left(T-t+\sqrt{T-t}\right)+\frac{L_{\sigma,z}}{1+L_{\sigma,z}}+L\sqrt{T-t}\right)\left\|(X^{0}-\tilde{X}^0,Y^{0}-\tilde{Y}^0,Z^{0}-\tilde{Z}^0)\right\|_w.
\end{multline*}

Note that the constant in front of $\left\|(X^{0}-\tilde{X}^0,Y^{0}-\tilde{Y}^0,Z^{0}-\tilde{Z}^0)\right\|_w$ converges
to $\frac{L_{\sigma,z}}{1+L_{\sigma,z}}<1$ for $t\rightarrow T$.
Furthermore, we obviously have
$$ Y^{1}_s-\tilde{Y}^{1}_s+\int_{s}^{T}(Z^{1}_r-\tilde{Z}^{1}_r)\dx W_r=\xi(X^{1}_T)-\xi(\tilde{X}^{1}_T)-\int_{s}^{T}\left(f(r,X^{1}_r,Y^{0}_r,Z^{0}_r)-f(r,\tilde{X}^{1}_r,\tilde{Y}^{0}_r,\tilde{Z}^{0}_r)\right)\dx r $$
and therefore
\begin{multline*}
\left(\mathbb{E}_{\hat{t}}\left[|Y^{1}_s-\tilde{Y}^{1}_s|^2\right]+\mathbb{E}_{\hat{t}}\left[\int_{s}^{T}|Z^{1}_r-\tilde{Z}^{1}_r|^2\dx r\right]\right)^\frac{1}{2}= \\
=\left(\mathbb{E}_{\hat{t}}\left[|Y^{1}_s-\tilde{Y}^{1}_s|^2\right]+\mathbb{E}_{\hat{t}}\left[\left|\int_{s}^{T}(Z^{1}_r-\tilde{Z}^{1}_r)\dx W_r\right|^2\right]\right)^\frac{1}{2}= \\
=\left(\mathbb{E}_{\hat{t}}\left[\left|Y^{1}_s-\tilde{Y}^{1}_s+\int_{s}^{T}(Z^{1}_r-\tilde{Z}^{1}_r)\dx W_r\right|^2\right]\right)^\frac{1}{2}\leq
\end{multline*}
$$ \leq \left(\mathbb{E}_{\hat{t}}\left[|\xi(X^{1}_T)-\xi(\tilde{X}^{1}_T)|^2\right]\right)^{\frac{1}{2}}+
\left(\mathbb{E}_{\hat{t}}\left[\left|\int_{s}^{T}f(r,X^{1}_r,Y^{0}_r,Z^{0}_r)-f(r,\tilde{X}^{1}_r,\tilde{Y}^{0}_r,\tilde{Z}^{0}_r)\dx r\right|^2\right]\right)^{\frac{1}{2}}\leq $$
\begin{multline*}
\leq L_{\xi,x}\left(\mathbb{E}_{\hat{t}}\left[|X^{1}_T-\tilde{X}^{1}_T|^2\right]\right)^{\frac{1}{2}}+L\left(T-t\right)\sup_{r\in[t,T]}\sqrt{\mathbb{E}_{\hat{t},\infty}\left[|X^{1}_r-\tilde{X}^1_r|^2\right]}+ \\
+L\left(T-t\right)\sup_{r\in[t,T]}\sqrt{\mathbb{E}_{\hat{t},\infty}\left[|Y^{0}_r-\tilde{Y}^0_r|^2\right]}+L\sqrt{T-t}\left(\mathbb{E}_{\hat{t},\infty}\left[\int_{t}^T|Z^{0}_r-\tilde{Z}^0_r|^2\dx r\right]\right)^{\frac{1}{2}}\leq
\end{multline*}
\begin{multline*}
\leq (L_{\xi,x}+L(T-t))\sup_{r\in[t,T]}\sqrt{\mathbb{E}_{\hat{t},\infty}\left[|X^{1}_r-\tilde{X}^1_r|^2\right]}+
L\left(T-t\right)\sup_{r\in[t,T]}\sqrt{\mathbb{E}_{\hat{t},\infty}\left[|Y^{0}_r-\tilde{Y}^0_r|^2\right]}+\\
+L\sqrt{T-t}\left(\mathbb{E}_{\hat{t},\infty}\left[\int_{t}^T|Z^{0}_r-\tilde{Z}^0_r|^2\dx r\right]\right)^{\frac{1}{2}}\leq
\end{multline*}
$$\leq (L_{\xi,x}+L(T-t))L\left(T-t+\sqrt{T-t}\right)\sup_{r\in[t,T]}\sqrt{\mathbb{E}_{\hat{t},\infty}\left[|X^{0}_r-\tilde{X}^0_r|^2\right]}+ $$
$$+\left((L_{\xi,x}+L(T-t))L\left(T-t+\sqrt{T-t}\right)+L\left(T-t\right)\right)\sup_{r\in[t,T]}\sqrt{\mathbb{E}_{\hat{t},\infty}\left[|Y^{0}_r-\tilde{Y}^0_r|^2\right]}+ $$
$$+\left((L_{\xi,x}+L(T-t))\left(L_{\sigma,z}+L\sqrt{T-t}\right)+L\sqrt{T-t}\right)\left(\mathbb{E}_{\hat{t},\infty}\left[\int_{t}^T|Z^{0}_r-\tilde{Z}^0_r|^2\dx r\right]\right)^{\frac{1}{2}}. $$
Finally, we obtain
$$(1+L_{\sigma,z})\left(\sup_{s\in[t,T]}\mathbb{E}_{\hat{t},\infty}\left[|Y^{1}_s-\tilde{Y}^{1}_s|^2\right]+\mathbb{E}_{\hat{t},\infty}\left[\int_{t}^{T}|Z^{1}_r-\tilde{Z}^{1}_r|^2\dx r\right]\right)^\frac{1}{2}\leq $$
$$\leq (1+L_{\sigma,z})(L_{\xi,x}+L(T-t))L\left(T-t+\sqrt{T-t}\right)\sup_{r\in[t,T]}\sqrt{\mathbb{E}_{\hat{t},\infty}\left[|X^{0}_r-\tilde{X}^0_r|^2\right]}+ $$
$$+\left((L_{\xi,x}+L(T-t))L\left(T-t+\sqrt{T-t}\right)+L\left(T-t\right)\right)(1+L_{\sigma,z})\sup_{r\in[t,T]}\sqrt{\mathbb{E}_{\hat{t},\infty}\left[|Y^{0}_r-\tilde{Y}^0_r|^2\right]}+ $$
$$+\left((L_{\xi,x}+L(T-t))\left(L_{\sigma,z}+L\sqrt{T-t}\right)+L\sqrt{T-t}\right)(1+L_{\sigma,z})\left(\mathbb{E}_{\hat{t},\infty}\left[\int_{t}^T|Z^{0}_r-\tilde{Z}^0_r|^2\dx r\right]\right)^{\frac{1}{2}}\leq $$
\begin{multline*}
\leq \Bigg((1+L_{\sigma,z})(L_{\xi,x}+L(T-t))L\left(T-t+\sqrt{T-t}\right)+ \\
+\left((L_{\xi,x}+L(T-t))L\left(T-t+\sqrt{T-t}\right)+L\left(T-t\right)\right)+ \\
+\left((L_{\xi,x}+L(T-t))\left(L_{\sigma,z}+L\sqrt{T-t}\right)+L\sqrt{T-t}\right)\Bigg)\left\|(X^{0}-\tilde{X}^0,Y^{0}-\tilde{Y}^0,Z^{0}-\tilde{Z}^0)\right\|_w.
\end{multline*}
Note that the constant in front of $\left\|(X^{0}-\tilde{X}^0,Y^{0}-\tilde{Y}^0,Z^{0}-\tilde{Z}^0)\right\|_w$ converges from above to the value
$L_{\sigma,z}\cdot L_{\xi,x}<1$ for $t\rightarrow T$.

We have finally shown
$$\left\|(X^{1}-\tilde{X}^1,Y^{1}-\tilde{Y}^1,Z^{1}-\tilde{Z}^1)\right\|_w\leq \gamma_t
\left\|(X^{0}-\tilde{X}^0,Y^{0}-\tilde{Y}^0,Z^{0}-\tilde{Z}^0)\right\|_w,$$
where
\begin{multline*} \gamma_t:=\left(2\cdot L\left(T-t+\sqrt{T-t}\right)+\frac{L_{\sigma,z}}{1+L_{\sigma,z}}+L\sqrt{T-t}\right)\,\vee\,\\ \vee\, \Bigg((1+L_{\sigma,z})(L_{\xi,x}+L(T-t))L\left(T-t+\sqrt{T-t}\right)+ \\
+\left((L_{\xi,x}+L(T-t))L\left(T-t+\sqrt{T-t}\right)+L\left(T-t\right)\right)+ \\
+\left((L_{\xi,x}+L(T-t))\left(L_{\sigma,z}+L\sqrt{T-t}\right)+L\sqrt{T-t}\right)\Bigg)
\end{multline*}
Note that $\gamma_t<1$ for $t<T$ large enough. More precisely $\lim_{t\uparrow T}\gamma_t=\frac{L_{\sigma,z}}{1+L_{\sigma,z}}\vee \left(L_{\sigma,z}\cdot L_{\xi,x}\right)$. Also note that $\gamma_t$ is monotonically decreasing in $t$.

If $\gamma_t<1$ holds we set $(X^0,Y^0,Z^0):=(0,0,0)$ and define recursively $$(X^k,Y^k,Z^k):=F(X^{k-1},Y^{k-1},Z^{k-1}),$$ for $k\in\mathbb{N}$. According to Banach's fixed point theorem this sequence converges in $\mathbb{G}_{\hat{t}}$ to a fixed point of $F$, which is unique.
This already shows existence and uniqueness of a $\mathbb{G}_{\hat{t}}$ - solution $(X,Y,Z)$ of the considered coupled FBSDE for a small interval.
Additionally, due to a priori estimates of the Banach fixed point theorem the norm $\|(X,Y,Z)\|_w$ is bounded by $\frac{1}{1-\gamma_t}\left\|(X^{1}-X^0,Y^{1}-Y^0,Z^{1}-Z^0)\right\|_w$ which in turn can be controlled in $Z$ by a bound which depends on $\gamma_t$, $\|X_t(\lambda,\cdot)\|_{\infty}$, $\left\|\left(|\mu|+|f|+|\sigma|\right)(\cdot,\cdot,0,0,0)\right\|_{\infty}$, $\|\xi(\cdot,0)\|_{\infty}$ and is monotonically increasing in these values. \\

Furthermore, we can show that $(X,Y,Z)$ is a progressively measurable function of $\lambda,s,\omega$ and that $\left(\frac{\dx}{\dx\lambda}X,\frac{\dx}{\dx\lambda}Y,\frac{\dx}{\dx\lambda}Z\right)$ exists and is in $\mathbb{H}$.\\
For this purpose define $(X^0,Y^0,Z^0):=(0,0,0)$ and recursively $(X^k,Y^k,Z^k):=F(X^{k-1},Y^{k-1},Z^{k-1})$.
We claim that for all $k$
\begin{itemize}
\item $X^{k},Y^{k},Z^{k}$ are progressively measurable and weakly differentiable w.r.t. $\lambda$ and
\item $(X^{k},Y^{k},Z^{k})\in\mathbb{H}$.
\end{itemize}
Clearly this holds for the index $k=0$. In order to run an inductive argument, assume that it holds up to an index $k-1$. We need to show that it also holds for $k$. In order to do this we consider $(X^k,Y^k,Z^k)=F(X^{k-1},Y^{k-1},Z^{k-1})$, which is really a system of two equations according to the definition of $F$ and differentiate it w.r.t. the parameter $\lambda$.\\
Using Lemmas \ref{wd2}, \ref{wd3} and \ref{chainrule} we obtain for all $v\in\mathbb{R}^n$:
\begin{multline}\label{firstequation}
\frac{\dx}{\dx\lambda}X^{k}_sv=\frac{\dx}{\dx\lambda}X_tv+\int_{t}^s\Delta^{(\ldots)}_{x}\mu\frac{\dx}{\dx\lambda}X^{k-1}_rv+\Delta^{(\ldots)}_{y}\mu\frac{\dx}{\dx\lambda}Y^{k-1}_rv+\Delta^{(\ldots)}_{z}\mu\frac{\dx}{\dx\lambda}Z^{k-1}_rv\dx r+\\
+\int_{t}^s\Delta^{(\ldots)}_{x}\sigma\frac{\dx}{\dx\lambda}X^{k-1}_rv+\Delta^{(\ldots)}_{y}\sigma\frac{\dx}{\dx\lambda}Y^{k-1}_rv+\Delta^{(\ldots)}_{z}\sigma\frac{\dx}{\dx\lambda}Z^{k-1}_rv\dx W_r,
\end{multline}
where "$\ldots$" stands for $(r,X^{k-1}_r,Y^{k-1}_r,Z^{k-1}_r)$. Also using Lemmas \ref{wd2}, \ref{wd5} and \ref{chainrule} we get for all $v\in\mathbb{R}^n$:
\begin{multline}\label{secondequation}
\frac{\dx}{\dx\lambda}Y^{k}_sv+\int_{s}^{T}\frac{\dx}{\dx\lambda}(Z^{k}_rv)\dx W_r=\Delta^{X^{k}_T}_{x}\xi \frac{\dx}{\dx\lambda}X^{k}_Tv-\\
-\int_{s}^{T}\Delta^{(\ldots)}_{x}f\frac{\dx}{\dx\lambda}X^{k}_rv+\Delta^{(\ldots)}_{y}f\frac{\dx}{\dx\lambda}Y^{k-1}_rv+\Delta^{(\ldots)}_{z}f\frac{\dx}{\dx\lambda}Z^{k-1}_rv\dx r.
\end{multline}
Here "$\ldots$" stands for $(r,X^{k}_r,Y^{k-1}_r,Z^{k-1}_r)$. \\
From (\ref{firstequation}) we deduce using Cauchy-Schwarz' and Minkowski's inequalities, as well as It\^o's isometry
\begin{multline*}
\left(\mathbb{E}_{\hat{t}}\left[\left|\frac{\dx}{\dx\lambda}X^{k}_s\right|_{v}^2\right]\right)^{\frac{1}{2}}\leq \left\|\frac{\dx}{\dx\lambda}X_t\right\|+ L\left(T-t+\sqrt{T-t}\right)\sup_{r\in[t,T]}\sqrt{\mathbb{E}_{\hat{t},\infty}\left[\left|\frac{\dx}{\dx\lambda}X^{k-1}_r\right|_{v}^2\right]}+\\
+L\left(T-t+\sqrt{T-t}\right)\sup_{r\in[t,T]}\sqrt{\mathbb{E}_{\hat{t},\infty}\left[\left|\frac{\dx}{\dx\lambda}Y^{k-1}_r\right|_{v}^2\right]}+\\
+\left(L_{\sigma,z}+L\sqrt{T-t}\right)\left(\mathbb{E}_{\hat{t},\infty}\left[\int_{t}^T\left|\frac{\dx}{\dx\lambda}Z^{k-1}_r\right|_{v}^2\dx r\right]\right)^{\frac{1}{2}}\leq
\end{multline*}
\begin{multline*}
\leq \left\|\frac{\dx}{\dx\lambda}X_t\right\|+L\left(T-t+\sqrt{T-t}\right)\sup_{r\in[t,T]}\sqrt{\mathbb{E}_{\hat{t},\infty}\left[\left|\frac{\dx}{\dx\lambda}X^{k-1}_r\right|_{v}^2\right]}+ \\
+L\left(T-t+\sqrt{T-t}\right)(1+L_{\sigma,z})\sup_{r\in[t,T]}\sqrt{\mathbb{E}_{\hat{t},\infty}\left[\left|\frac{\dx}{\dx\lambda}Y^{k-1}_r\right|_{v}^2\right]}+ \\
+\frac{L_{\sigma,z}+L\sqrt{T-t}}{1+L_{\sigma,z}}(1+L_{\sigma,z})\left(\mathbb{E}_{\hat{t},\infty}\left[\int_{t}^T\left|\frac{\dx}{\dx\lambda}Z^{k-1}_r\right|_{v}^2\dx r\right]\right)^{\frac{1}{2}}\leq
\end{multline*}
$$\leq \left\|\frac{\dx}{\dx\lambda}X_t\right\|+\left(2L\left(T-t+\sqrt{T-t}\right)+
\frac{L_{\sigma,z}}{1+L_{\sigma,z}}+L\sqrt{T-t}\right)
\left\|\left(\frac{\dx}{\dx\lambda}X^{k-1},\frac{\dx}{\dx\lambda}Y^{k-1},\frac{\dx}{\dx\lambda}Z^{k-1}\right)\right\|_s, $$
where $\left\|\frac{\dx}{\dx\lambda}X_t\right\|:=\esssup_{\lambda\in\mathbb{R}^n}\sup_{v\in S^{n-1}}\mathbb{E}_{\hat{t},\infty}\left[\left|\frac{\dx}{\dx \lambda}X_t(\lambda,\cdot)\right|_{v}\right]$.

Similarly equation (\ref{secondequation}) implies
$$ \left(\mathbb{E}_{\hat{t}}\left[\left|\frac{\dx}{\dx\lambda}Y^{k}_s\right|_v^2\right]+
\mathbb{E}_{\hat{t}}\left[\int_{s}^{T}\left|\frac{\dx}{\dx\lambda}Z^{k}_r\right|_v^2\dx r\right]\right)^{\frac{1}{2}}\leq $$
$$\leq (L_{\xi,x}+L(T-t))\sup_{r\in[t,T]}\sqrt{\mathbb{E}_{\hat{t},\infty}\left[\left|\frac{\dx}{\dx\lambda}X^{k}_r\right|_v^2\right]}+
L\left(T-t\right)\sup_{r\in[t,T]}\sqrt{\mathbb{E}_{\hat{t},\infty}\left[\left|\frac{\dx}{\dx\lambda}Y^{k-1}_r\right|_v^2\right]}+$$
$$+L\sqrt{T-t}\left(\mathbb{E}_{\hat{t},\infty}\left[\int_{t}^T\left|\frac{\dx}{\dx\lambda}Z^{k-1}_r\right|_v^2\dx r\right]\right)^{\frac{1}{2}}\leq $$
\begin{multline*}
\leq (L_{\xi,x}+L(T-t))\cdot \left\|\frac{\dx}{\dx\lambda}X_t\right\|+ \\
(L_{\xi,x}+L(T-t))L\left(T-t+\sqrt{T-t}\right)\sup_{r\in[t,T]}\sqrt{\mathbb{E}_{\hat{t},\infty}\left[\left|\frac{\dx}{\dx\lambda}X^{k-1}_r\right|_v^2\right]}+ \\
+\left((L_{\xi,x}+L(T-t))L\left(T-t+\sqrt{T-t}\right)+L\left(T-t\right)\right)\sup_{r\in[t,T]}\sqrt{\mathbb{E}_{\hat{t},\infty}\left[\left|\frac{\dx}{\dx\lambda}Y^{k-1}_r\right|_v^2\right]}+ \\
+\left((L_{\xi,x}+L(T-t))\left(L_{\sigma,z}+L\sqrt{T-t}\right)+L\sqrt{T-t}\right)\left(\mathbb{E}_{\hat{t},\infty}\left[\int_{t}^T\left|\frac{\dx}{\dx\lambda}Z^{k-1}_r\right|_v^2\dx r\right]\right)^{\frac{1}{2}}\leq
\end{multline*}
$$\leq \frac{1}{1+L_{\sigma,z}}\Bigg((L_{\xi,x}+L(T-t))\cdot \left\|\frac{\dx}{\dx\lambda}X_t\right\|+(1+L_{\sigma,z})(L_{\xi,x}+L(T-t))L\left(T-t+\sqrt{T-t}\right)+ $$
$$+\left((L_{\xi,x}+L(T-t))L\left(T-t+\sqrt{T-t}\right)+L\left(T-t\right)\right)+ $$
$$+\left((L_{\xi,x}+L(T-t))\left(L_{\sigma,z}+L\sqrt{T-t}\right)+L\sqrt{T-t}\right)\Bigg)
\left\|\left(\frac{\dx}{\dx\lambda}X^{k-1},\frac{\dx}{\dx\lambda}Y^{k-1},\frac{\dx}{\dx\lambda}Z^{k-1}\right)\right\|_s.$$
This leads to
\begin{multline*}
\left\|\left(\frac{\dx}{\dx\lambda}X^{k},\frac{\dx}{\dx\lambda}Y^{k},\frac{\dx}{\dx\lambda}Z^{k}\right)\right\|_s\leq \\
\leq\left\|\frac{\dx}{\dx\lambda}X_t\right\|\vee\left((L_{\xi,x}+L(T-t))\left\|\frac{\dx}{\dx\lambda}X_t\right\|\right)+
\gamma_t\left\|\left(\frac{\dx}{\dx\lambda}X^{k-1},\frac{\dx}{\dx\lambda}Y^{k-1},\frac{\dx}{\dx\lambda}Z^{k-1}\right)\right\|_s.
\end{multline*}
As mentioned $\gamma_t\leq \gamma_{t'}<1$ for all $t\in[t',T)$, if $t'$ is large enough. For such $t$ we have
\begin{equation}\label{bb}\sup_{k\in\mathbb{N}_0}\left\|\left(\frac{\dx}{\dx\lambda}X^{k},\frac{\dx}{\dx\lambda}Y^{k},\frac{\dx}{\dx\lambda}Z^{k}\right)\right\|_s\leq \frac{\left\|\frac{\dx}{\dx\lambda}X_t\right\|\vee(L_{\xi,x}+L(T-t))\left\|\frac{\dx}{\dx\lambda}X_t\right\|}{1-\gamma_{t'}}\leq \left\|\frac{\dx}{\dx\lambda}X_t\right\|\cdot K,\end{equation}
where $$K:=\frac{1\vee(L_{\xi,x}+L(T-t))}{1-\frac{L_{\sigma,z}}{1+L_{\sigma,z}}\vee \left(L_{\sigma,z}\cdot L_{\xi,x}\right)}<\infty.$$

Knowing that $\left(X^{k},Y^{k},Z^{k}\right)$ converges to $\left(X,Y,Z\right)$ this already implies by Lemma \ref{weakderivexist} that the weak derivative $\left(\frac{\dx}{\dx\lambda}X,\frac{\dx}{\dx\lambda}Y,\frac{\dx}{\dx\lambda}Z\right)$ exists and satisfies
\begin{equation}\label{bbb}\left\|\left(\frac{\dx}{\dx\lambda}X,\frac{\dx}{\dx\lambda}Y,\frac{\dx}{\dx\lambda}Z\right)\right\|_s
\leq C K \left\|\frac{\dx}{\dx\lambda}X_t\right\|,\end{equation}
for all $t\in[t',T]$, where $C<\infty$ is some constant (not depending on $t$). Here Lemma \ref{weakderivexist} is applied to each component of $X,Y$ and $Z$ separately and in case of $X$ and $Y$ for each fixed time $s\in[t,T]$ separately. Furthermore, in order to apply Lemma \ref{weakderivexist}, which is formulated for integrals and not for conditional expecations, we need to decompose $\Omega$ into $\Omega_1\times\Omega_2$ such that the first component in $\omega=(\omega_1,\omega_2)$ represents all the information until time $\hat{t}$ and the second the remaining information. Then we can fix $\omega_1$ and write for instance $\mathbb{E}_{\hat{t}}[Y_s](\omega_1,\omega_2)=\mathbb{E}[Y_s|\mathcal{F}_{\hat{t}}](\omega_1,\omega_2)=\mathbb{E}[Y_s(\omega_1,\cdot)]$, etc. So Lemma \ref{weakderivexist} can be applied for each fixed $\omega_1$ separately. Also note that norms of the form $\sup_{v\in S^{n-1}}\sqrt{\mathbb{E}[|\cdot |^2_v]}$ are equivalent to norms $\sqrt{\mathbb{E}[|\cdot |^2]}$.

Moreover, we can deduce a more restrictive bound for $\frac{\dx}{\dx\lambda}Y$. Let us write $\tilde{K}:=\left\|\frac{\dx}{\dx\lambda}X_t\right\|\cdot K$ for short.
Using (\ref{secondequation}) and (\ref{bb}) we have
$$ \left(\mathbb{E}_{\hat{t}}\left[\left|\frac{\dx}{\dx\lambda}Y^k_s\right|_v^2\right]+
\mathbb{E}_{\hat{t}}\left[\int_{s}^{T}\left|\frac{\dx}{\dx\lambda}Z^k_r\right|_v^2\dx r\right]\right)^{\frac{1}{2}}
\leq L_{\xi,x}\sup_{r\in[t,T]}\sqrt{\mathbb{E}_{\hat{t}}\left[\left|\frac{\dx}{\dx\lambda}X^k_r\right|_v^2\right]}+
C\tilde{K}\sqrt{T-t},$$
where $C$ is some constant not depending on $t$ or $X_t$. In other words
$$ \mathbb{E}_{\hat{t}}\left[\left|\frac{\dx}{\dx\lambda}Y^k_s\right|_v^2\right]+
\mathbb{E}_{\hat{t}}\left[\int_{s}^{T}\left|\frac{\dx}{\dx\lambda}Z^k_r\right|_v^2\dx r\right]
\leq L_{\xi,x}^2\sup_{r\in[t,T]}\mathbb{E}_{\hat{t}}\left[\left|\frac{\dx}{\dx\lambda}X^k_r\right|_v^2\right]+
C\tilde{K}^2(\sqrt{T-t}+(T-t)).$$
Letting $k\rightarrow\infty$ we get
$$ \mathbb{E}_{\hat{t}}\left[\left|\frac{\dx}{\dx\lambda}Y_s\right|_v^2\right]+
\mathbb{E}_{\hat{t}}\left[\int_{s}^{T}\left|\frac{\dx}{\dx\lambda}Z_r\right|_v^2\dx r\right]
\leq L_{\xi,x}^2\sup_{r\in[t,T]}\mathbb{E}_{\hat{t}}\left[\left|\frac{\dx}{\dx\lambda}X_r\right|_v^2\right]+
C\tilde{K}^2(\sqrt{T-t}+(T-t)).$$
Furthermore, we observe
$$\mathbb{E}_{\hat{t}}\left[\left|\frac{\dx}{\dx\lambda}X^k_s\right|_v^2\right]\leq \left\|\frac{\dx}{\dx\lambda}X_t\right\|^2+\left\|\frac{\dx}{\dx\lambda}X_t\right\|C\tilde{K}\sqrt{T-t}+C\tilde{K}^2\sqrt{T-t}+
L_{\sigma,z}^2\mathbb{E}_{\hat{t}}\left[\int_{t}^T\left|\frac{\dx}{\dx\lambda}Z^{k-1}_r\right|_v^2\dx r\right],$$
for all $s\in[t,T]$.
This can be seen by taking to the squares both sides of (\ref{firstequation}), writing the right hand side as a sum of products, taking expectations and using Cauchy-Schwarz' inequality together with (\ref{bb}) several times. \\
Letting $k\rightarrow\infty$ leads to
$$\mathbb{E}_{\hat{t}}\left[\left|\frac{\dx}{\dx\lambda}X_s\right|_v^2\right]\leq \left\|\frac{\dx}{\dx\lambda}X_t\right\|^2+\left\|\frac{\dx}{\dx\lambda}X_t\right\|C\tilde{K}\sqrt{T-t}+C\tilde{K}^2\sqrt{T-t}+
L_{\sigma,z}^2\mathbb{E}_{\hat{t}}\left[\int_{t}^T\left|\frac{\dx}{\dx\lambda}Z_r\right|_v^2\dx r\right],$$
for all $\hat{t}\in[t,T]$.
By plugging this last inequality into the preceding inequality for $Y,Z$, we have
$$ \mathbb{E}_{\hat{t}}\left[\left|\frac{\dx}{\dx\lambda}Y_s\right|_v^2\right]+
\mathbb{E}_{\hat{t}}\left[\int_{s}^{T}\left|\frac{\dx}{\dx\lambda}Z_r\right|_v^2\dx r\right]\leq$$
$$\leq L_{\xi,x}^2\left\|\frac{\dx}{\dx\lambda}X_t\right\|^2+(L_{\xi,x}L_{\sigma,z})^2\mathbb{E}_{\hat{t}}\left[\int_{t}^T\left|\frac{\dx}{\dx\lambda}Z_r\right|_v^2\dx r\right]+
C\tilde{K}\left(\tilde{K}+\left\|\frac{\dx}{\dx\lambda}X_t\right\|\right)(\sqrt{T-t}+(T-t)).$$
Because of $L_{\xi,x}L_{\sigma,z}<1$ we have
\begin{multline}\label{bbbb}\left(\mathbb{E}_{\hat{t}}\left[\left|\frac{\dx}{\dx\lambda}Y_s\right|_v^2\right]\right)^\frac{1}{2}\leq L_{\xi,x}\left\|\frac{\dx}{\dx\lambda}X_t\right\|+\sqrt{C\tilde{K}\left(\tilde{K}+\left\|\frac{\dx}{\dx\lambda}X_t\right\|\right)\left(\sqrt{T-t}+(T-t)\right)}=\\
=\left\|\frac{\dx}{\dx\lambda}X_t\right\|\cdot\left(L_{\xi,x}+\sqrt{CK\left(K+1\right)\left(\sqrt{T-t}+(T-t)\right)}\right).\end{multline}

Now, for any $\lambda\in\mathbb{R}^n$ and any $t\in[t',T)$ set
$$u(t,\cdot,\lambda):=Y_t(\lambda,\cdot),$$
 where $(X,Y,Z)$ is the unique $\mathbb{G}_{t}$ - solution to the FBSDE considered above with $X_t(\lambda,\omega):=\lambda$. Note that $u(t,\cdot,\lambda)$ is $\mathcal{F}_t$-measurable. \\
Using inequality (\ref{bbbb}) and $\left\|\frac{\dx}{\dx\lambda}X_t\right\|=1$
we obtain that $u(t,\cdot)$ is Lipschitz continuous in $\lambda$ with a Lipschitz constant, which can be bounded away from $\frac{1}{L_{\sigma,z}}$ (a value strictly larger than $L_{\xi,x}$) by choosing $t'<T$ large enough. \\
Also note that $\|u(\cdot,\cdot,0)\|_{\infty}<\infty$ holds due to the aforementioned a priori bound on $\|(X,Y,Z)\|_w$.

Furthermore observe that for some random $\mathcal{F}_t$-measurable $X_t:\Omega\rightarrow\mathbb{R}^n$ with $\mathbb{E}_{\hat{t},\infty}\left[|X_t|^2\right]<\infty$ the corresponding $X,Y,Z$ would have to satisfy $Y_t=u(t,X_t)$ a.s.: This can be shown, by assuming without loss of generality that $\Omega=\Omega_1\times\Omega_2$ where the projections $\pi_1$, $\pi_2$ on the two components are independent such that $\mathcal{F}_t=\sigma(\pi_1)$ and $\sigma((W_r-W_t)_{r\in[T,t]})=\sigma(\pi_2)$ and so $X_t$ can be assumed to be a function of $\omega_1$. Now fix the first component $\omega_1$, so $X_t(\omega_1)$ becomes a constant and $X,Y,Z$ only depend on the second component $\omega_2$ and solve some Lipschitz FBSDE on $[t,T]$, which is the same FBSDE that is solved by processes $X^\lambda,Y^\lambda,Z^\lambda$ obtained from the above problem with initial value $\lambda\in\mathbb{R}^n$ if we fix the first component $\omega_1$ and choose $\lambda:=X_t(\omega_1)$. Remember $u(t,\omega_1,\lambda):=Y^\lambda_t(\omega_1)$. Note also that $(X,Y,Z)$ and $(X^\lambda,Y^\lambda,Z^\lambda)$ are both in $\mathbb{G}_{\hat{t}}$. If we fix $\omega_1$ they will be both in an analogous space, which, like $\mathbb{G}_{\hat{t}}$, will have the property that two solutions to the same Lipschitz FBSDE with the same Lipschitz constants as that of $\mu,\sigma,f,\xi$ (or smaller) must coincide if both solutions are in this $\mathbb{G}_{\hat{t}}$ - like space. This shows $Y_t(\omega_1)=Y^\lambda_t(\omega_1)=u(t,\omega_1,\lambda)=u(t,\omega_1,X_t(\omega_1))$.

If we start at some $\mathcal{F}_t$-measurable $X_t$ s.t. $\mathbb{E}_{\hat{t},\infty}\left[|X_t|^2\right]<\infty$, $t\in[t',T)$, and consider the corresponding $X_s$, $s\in[t,T]$, which is $\mathcal{F}_s$-measurable, we will have $Y_s=u(s,X_s)$, since the same argument as above can be applied on the interval $[s,T]$ starting with $X_s$ and decomposing $\Omega$ in an $\mathcal{F}_s$ - component and an independent $\sigma(W_r-W_s,r\in[s,T])$ - component. We again have used that all processes considered are in a sufficiently strongly regular ($\mathbb{G}_{\hat{t}}$ - like) space.

Let $\tilde{X}^{(t)},\tilde{Y}^{(t)},\tilde{Z}^{(t)}$ be processes on $\mathbb{R}^n\times [t,T]\times \Omega$ as constructed above (via Picard iteration) with initial condition $X_t(\lambda,\omega):=\lambda$, $(\lambda,\omega)\in\mathbb{R}^n\times\Omega$. For every  $\lambda\in\mathbb{R}^n$ and $s\in [t,T]$ we have $\tilde{Y}^{(t)}_s(\lambda,\cdot)=u(s,\cdot,\tilde{X}^{(t)}_s(\lambda,\cdot))$ a.s. as mentioned earlier.

Now we can show that $u:[t',T]\times\mathbb{R}^n\longrightarrow\mathbb{R}^m$ is a decoupling field. Choose any $t_1<t_2$ from $[t',T]$ and \emph{any} $\mathcal{F}_{t_1}$-measurable initial condition $X_{t_1}$. Define $X_s(\omega):=\tilde{X}^{(t_1)}_s(X_{t_1}(\omega),\omega)$, $Y_s(\omega):=\tilde{Y}^{(t_1)}_s(X_{t_1}(\omega),\omega)$, $Z_s(\omega):=\tilde{Z}^{(t_1)}_s(X_{t_1}(\omega),\omega)$ for all $s\in [t_1,t_2]$ and $\omega\in\Omega$. It is straightforward to check that $X,Y,Z$ are progressively measurable, satisfy the FBSDE and the decoupling condition (all these properties are inherited from $\tilde{X}^{(t_1)},\tilde{Y}^{(t_1)},\tilde{Z}^{(t_1)}$). The initial condition is satisfied via definition of $X$ and $\tilde{X}^{(t_1)}$. \\

We can also show a.e.-uniqueness of the processes $X,Y,Z$ on $[t,t_2]\times\Omega\times\mathbb{R}^n$ solving the FBSDE together with the decoupling condition via $u$ and intial condition $X_t=x\in\mathbb{R}^n$ , where $[t,t_2]\subseteq [t',T]$.

The triplets $(X,Y,Z)$ constructed so far are in $\mathbb{G}_{t}$. Assume that there is another triplet
$(\hat{X},\hat{Y},\hat{Z})$ with the above properties. If we can show that this triplet must be in $\mathbb{G}_{t}$, we are done. Otherwise observe that for every stopping time $\tau\in[t,t_2]$ the triplets $(X_{\cdot\wedge\tau},Y_{\cdot\wedge\tau},Z\mathbf{1}_{\{\cdot\leq\tau\}})$ and
$(\hat{X}_{\cdot\wedge\tau},\hat{Y}_{\cdot\wedge\tau},\hat{Z}\mathbf{1}_{\{\cdot\leq\tau\}})$ both solve the FBSDE
given by $\hat{\mu}=\mu\mathbf{1}_{\{\cdot\leq\tau\}}$, $\hat{\sigma}=\sigma\mathbf{1}_{\{\cdot\leq\tau\}}$,
$\hat{f}=f\mathbf{1}_{\{\cdot\leq\tau\}}$ and $\hat{\xi}=u(\tau,\cdot)$. Note that $L_{\hat{\xi},x}\leq L_{u,x}<L_{\sigma,z}^{-1}\leq L_{\hat{\sigma},z}^{-1}$. This new FBSDE on $[t,t_2]$ has the same properties as the initial one and we also have uniqueness of $\mathbb{G}_{t}$ - solutions. If $\tau$ is chosen such that
$(\hat{X}_{\cdot\wedge\tau},\hat{Y}_{\cdot\wedge\tau},\hat{Z}\mathbf{1}_{\{\cdot\leq\tau\}})$ is in $\mathbb{G}_{t}$,
$(\hat{X},\hat{Y},\hat{Z})$ and $(X,Y,Z)$ will have to coincide on $[t,\tau]$. Using localization, we see that
$(\hat{X},\hat{Y},\hat{Z})$ and $(X,Y,Z)$ must be a.e. equal.

Uniqueness of $u$ on $[t',T]$ follows easily from our knowledge, that processes $(X,Y,Z)$ associated with decoupling fields with the properties $L_{u,x}<L_{\sigma,z}^{-1}$ and $\|u(\cdot,\cdot,0)\|_{\infty}<\infty$ are always in $\mathbb{G}_{t}$, as we have seen (at least if $X_t=x\in\mathbb{R}^n$): Since the FBSDE on an interval $[t,T]$, $t\in[t',T]$ is the same for all decoupling fields, the $\mathbb{G}_{t}$ - solution $X,Y,Z$ is also the same for all decoupling fields. This uniquely determines $u(t,\cdot,x)=Y_t(x,\cdot)$.
\end{proof}

\begin{rem}\label{hchoice}
We observe from the proof that the supremum of all $h=T-t$, with $t$ satisfying the properties required in Theorem \ref{locexist} can be bounded away from $0$ by a bound, which only depends on
\begin{itemize}
\item the Lipschitz constants of $\mu,\sigma$ and $f$ w.r.t. to the last $3$ components,
\item $L_{\xi}$ and $L_{\xi}\cdot L_{\sigma,z}$,
\end{itemize}
and which is monotonically decreasing in these values.
\end{rem}

\begin{rem} \label{bbbbrem}
As we have seen (\ref{bbbb}) implies that our decoupling field $u$ on $[t,T]$ satisfies
$$ L_{u,x}\leq L_{\xi,x}+C(T-t)^{\frac{1}{4}},$$
where $C$ is some constant which does not depend on $t$.
\end{rem}

\begin{rem} \label{xyz}
If we do not care about decoupling fields but are only interested in processes $X,Y,Z$ solving the forward and the backward equations together with $Y_T=\xi(X_T)$ for given $t_1\in[t,T],t_2:=T,X_{t_1}$, the above construction does provide existence but not uniqueness.

In order to have uniqueness we need an additional restriction, e.g. $(X,Y,Z)\in \mathbb{G}_{0}$. Under this condition we would get not only uniqueness but also the decoupling condition $Y_s=u(s,X_s)$.

Conversely the two conditions $X_{t_1}=x\in\mathbb{R}^n$ and $Y_s=u(s,X_s)$ would also suffice for uniqueness as we have seen (in fact we saw that this implies $(X,Y,Z)\in \mathbb{G}_{t_1}\subseteq\mathbb{G}_{0}$).

\end{rem}

\section{Some examples} \label{examples}
We first demonstrate that the assumption $L_{\sigma,z}\cdot L_{\xi,x}<1$ cannot be dropped or weakened to a larger bound.
\begin{exa} For instance, consider the forward backward problem
$$ X_t=x_0+\int_0^t(\sigma_0+Z_s)\dx W_s, $$
$$ Y_t=X_T-\int_t^TZ_s\dx W_s. $$
This means that $\mu=f=0$ and $\sigma(s,x,y,z)=\sigma_0+z$, where $\sigma_0\in\mathbb{R}\backslash\{0\}$, $\xi=\mathrm{Id}_{\mathbb{R}}$. We also assume that $X,Y,Z,W$ must all be real-valued. Observe $L_{\xi,x}=1$, $L_{\sigma,z}=1$ and therefore $L_{\xi,x}L_{\sigma,z}=1$. We claim that this problem cannot have an adapted solution, no matter how small $T$ is.\\
In fact, the forward equation implies
$$ X_T-X_t=\int_t^T(\sigma_0+Z_s)\dx W_s=\sigma_0(W_T-W_t)+\int_t^TZ_s\dx W_s $$
or
$$ X_T-\int_t^TZ_s\dx W_s=X_t+\sigma_0(W_T-W_t). $$
Together with the backward equation we obtain
$$ Y_t=X_t+\sigma_0(W_T-W_t)$$
which would mean
$$ Y_0-x_0=\sigma_0 W_T.$$
This cannot be true since $Y_0-x_0$ is a.s. constant and $\sigma_0 W_T$ is a non-degenerate Gaussian random variable.
\end{exa}
The requirement to choose $T-t$ small enough cannot be omitted either.
\begin{exa}
For $t\in[0,T)$ consider the following FBSDE on the interval $[t,T]$:
$$ X_s=x+\int_t^s Y_r\dx r, $$
$$ Y_s=X_T-\int_s^T Z_r\dx W_r, \quad s\in[t,T].$$
For $1>T-t$ the problem has a decoupling field
$$ u(s,x)=\frac{x}{1-(T-s)},\quad s\in [t,T], $$
such that
$$ X_s=x+(s-t)\frac{x}{1-(T-t)}=x\frac{1-(T-s)}{1-(T-t)}, $$
$$ Y_s=\frac{x}{1-(T-t)}, $$
$$ Z_s=0. $$
We will see later that this $u$ is unique among all decoupling fields on $[t,T]$ with $L_{u,x}<\infty$ and $\|u(\cdot,\cdot,0)\|_{\infty}<\infty$. \\
Note that $u(t,x)$ tends to infinity for $t\downarrow (T-1)$ and $x\neq 0$. Thus there is no decoupling field on $[T-1,T]$ with $L_{u,x}<\infty$ and $\|u(\cdot,\cdot,0)\|_{\infty}<\infty$. Note that here $L_{\sigma,z} = 0$.
\end{exa}

\section{Regularity} \label{regularity}

\begin{defi}
Let $u:[t,T]\times\Omega\times\mathbb{R}^n\rightarrow\mathbb{R}^m$ be a decoupling field to $\fbsde(\xi,(\mu,\sigma,f))$. We call $u$ \emph{weakly regular}, if
$L_{u,x}<L_{\sigma,z}^{-1}$ and $\|u(\cdot,\cdot,0)\|_{\infty}<\infty$.
\end{defi}

The decoupling field constructed in Theorem \ref{locexist} is weakly regular as we have seen. In practice however it is important to have explicit knowledge about the regularity of $(X,Y,Z)$. For instance, it is important to know in which spaces the processes live, and how they react to changes in the initial value. Specifically it can be very useful to have differentiability of $X,Y,Z$ w.r.t. the initial value.

\begin{defi}
Let $u:[t,T]\times\Omega\times\mathbb{R}^n\rightarrow\mathbb{R}^m$ be a weakly regular decoupling field to $\fbsde(\xi,(\mu,\sigma,f))$. We call $u$ \emph{strongly regular} if for all fixed $t_1,t_2\in[t,T]$, $t_1\leq t_2,$ the processes $X,Y,Z$ arising in the defining property of a decoupling field
are a.e. unique for each \emph{constant} initial value $X_{t_1}=x\in\mathbb{R}^n$ and satisfy
$$\sup_{s\in [t_1,t_2]}\mathbb{E}_{t_1,\infty}[|X_s|^2]+\sup_{s\in [t_1,t_2]}\mathbb{E}_{t_1,\infty}[|Y_s|^2]
+\mathbb{E}_{t_1,\infty}\left[\int_{t_1}^{t_2}|Z_s|^2\dx s\right]<\infty\quad\forall x\in\mathbb{R}^n.$$
In addition they must be measurable as functions of $(x,s,\omega)$ and
even weakly differentiable w.r.t. $x$ such that
$$\esssup_{x\in\mathbb{R}^n}\sup_{v\in S^{n-1}}\sup_{s\in [t_1,t_2]}\mathbb{E}_{t_1,\infty}\left[\left|\frac{\dx}{\dx x}X_s\right|^2_v\right]<\infty,$$
$$\esssup_{x\in\mathbb{R}^n}\sup_{v\in S^{n-1}}\sup_{s\in [t_1,t_2]}\mathbb{E}_{t_1,\infty}\left[\left|\frac{\dx}{\dx x}Y_s\right|^2_v\right]<\infty,$$
$$\esssup_{x\in\mathbb{R}^n}\sup_{v\in S^{n-1}}\mathbb{E}_{t_1,\infty}\left[\int_{t_1}^{t_2}\left|\frac{\dx}{\dx x}Z_s\right|^2_v\dx s\right]<\infty.$$
We say that a decoupling field on $[t,T]$ is \emph{strongly regular} on a subinterval $[t_1,t_2]\subseteq[t,T]$ if $u$ restricted to $[t_1,t_2]$ is a strongly regular decoupling field for $\fbsde(u(t_2,\cdot),(\mu,\sigma,f))$.
\end{defi}

\begin{rem}\label{locexistext}
We have seen in the proof of Theorem \ref{locexist} that the constructed decoupling field $u$ is strongly regular (recall the definitions of $\|\cdot\|_w$ and $\|\cdot\|_s$ in the proof, and  replace $\lambda$ with $x$).
\end{rem}

\begin{lem}\label{gluestrongly regular}
Let $g,\mu,\sigma,f$ be as in Theorem \ref{locexist}, let $0\leq s<t<T$ and let $u$ be a weakly regular decoupling field for $\fbsde (\xi,(\mu,\sigma,f))$ on $[s,T]$.\\
If $u$ is strongly regular on $[s,t]$ and $T-t$ is small enough as required in Theorem \ref{locexist} resp. Remark \ref{hchoice}, then $u$ is strongly regular on $[s,T]$.
\end{lem}
\begin{proof}
We only need to demonstrate the strong regularity of $u$ for the case $s\leq t_1\leq t\leq t_2=T$.
The processes $X,Y,Z$ corresponding to the interval $[t_1,T]$ with initial value $X_{t_1}=x\in\mathbb{R}^n$ will obviously be
unique at least on $[t_1,t]$ and satisfy
$$\sup_{s\in [t_1,t]}\mathbb{E}_{t_1,\infty}[|X_s|^2]+\sup_{s\in [t_1,t]}\mathbb{E}_{t_1,\infty}[|Y_s|^2]
+\mathbb{E}_{t_1,\infty}\left[\int_{t_1}^{t}|Z_s|^2\dx s\right]<\infty\quad\forall x\in\mathbb{R}^n,$$
$$\esssup_{x\in\mathbb{R}^n}\sup_{v\in S^{n-1}}\sup_{r\in [t_1,t]}\mathbb{E}_{t_1,\infty}\left[\left|\frac{\dx}{\dx x}X_r\right|^2_v\right]<\infty,$$
$$\esssup_{x\in\mathbb{R}^n}\sup_{v\in S^{n-1}}\sup_{r\in [t_1,t]}\mathbb{E}_{t_1,\infty}\left[\left|\frac{\dx}{\dx x}Y_r\right|^2_v\right]<\infty,$$
$$\esssup_{x\in\mathbb{R}^n}\sup_{v\in S^{n-1}}\mathbb{E}_{t_1,\infty}\left[\int_{t_1}^{t}\left|\frac{\dx}{\dx x}Z_r\right|^2_v\dx r\right]<\infty.$$
In particular $\esssup_{x\in\mathbb{R}^n}\sup_{v\in S^{n-1}}\mathbb{E}_{t_1,\infty}\left[\left|\frac{\dx}{\dx x}X_t\right|^2_v\right]<\infty$ and also $\mathbb{E}_{t_1,\infty}[|X_t|^2]<\infty$ for all $x\in\mathbb{R}^n$.
The map $(x,\omega)\mapsto X_t(x,\omega)$ must be measurable since $(x,s,\omega)\mapsto X_s(x,\omega)$ is measurable and $X$ is continuous in time. \\
Now according to the proof of Theorem \ref{locexist} construct for this $X_t:\mathbb{R}^n\times\Omega\rightarrow\mathbb{R}^n$ progressive $\check{X},\check{Y},\check{Z}$ on $\mathbb{R}^n\times [t,T]\times\Omega$ such that
$$\sup_{s\in [t,T]}\mathbb{E}_{t_1,\infty}[|\check{X}_s|^2]+\sup_{s\in [t,T]}\mathbb{E}_{t_1,\infty}[|\check{Y}_s|^2]
+\mathbb{E}_{t_1,\infty}\left[\int_{t}^{T}|\check{Z}_s|^2\dx s\right]<\infty\quad\forall x\in\mathbb{R}^n,$$
$$\esssup_{x\in\mathbb{R}^n}\sup_{v\in S^{n-1}}\sup_{r\in [t,T]}\mathbb{E}_{t_1,\infty}\left[\left|\frac{\dx}{\dx x}\check{X}_r\right|^2_v\right]<\infty,$$
$$\esssup_{x\in\mathbb{R}^n}\sup_{v\in S^{n-1}}\sup_{r\in [t,T]}\mathbb{E}_{t_1,\infty}\left[\left|\frac{\dx}{\dx x}\check{Y}_r\right|^2_v\right]<\infty,$$
$$\esssup_{x\in\mathbb{R}^n}\sup_{v\in S^{n-1}}\mathbb{E}_{t_1,\infty}\left[\int_{t}^{T}\left|\frac{\dx}{\dx x}\check{Z}_r\right|^2_v\dx r\right]<\infty,$$
and such that $\check{X},\check{Y},\check{Z}$ solve $\fbsde (\xi,(\mu,\sigma,f))$ on $[t,T]$ with $\check{Y}_r=u(r,\check{X}_r)$ and $\check{X}_t=X_t$ a.s.. We have seen in the proof of Theorem \ref{locexist} that such $\check{X},\check{Y},\check{Z}$ are unique. \\
As just mentioned $\check{X},\check{Y},\check{Z}$ are progressive, solve $\fbsde (\xi,(\mu,\sigma,f))$ on $[t,T]$ with $\check{Y}_r=u(r,\check{X}_r)$ and $\check{X}_t=X_t$ a.s.. These properties are also satisfied by $X,Y,Z$ on $[t,T]$. In the proof of Theorem \ref{locexist} we have seen that this implies $(\check{X},\check{Y},\check{Z})=(X,Y,Z)$ on $[t,T]$. This proves strong regularity of $u$.
\end{proof}

\section{Extension to large intervals} \label{global}

In the following we employ local results from the previous sections to obtain global existence and uniqueness in a sense specified later. We will extensively use a simple basic argument which we will refer to as \emph{small interval induction}.

\begin{lem}[Small interval induction, backward]
\label{sii}
Let $T_1<T_2$ be real numbers and let $S\subseteq [T_1,T_2]$ s.t.
\begin{itemize}
\item $T_2\in S$,
\item there exists an $h>0$ s.t. $[s-h,s]\cap [T_1,T_2]\subseteq S$ for all $s\in S$.
\end{itemize}
Then $S=[T_1,T_2]$. In particular $T_1\in S$.
\end{lem}
\begin{proof}
Let $s_{\mathrm{min}}$ be the infimum of all $s\in S$ such that $[s,T_2]\subseteq S$. Obviously $(s_{\mathrm{min}},T_2]\subseteq S$. We claim that $s_{\mathrm{min}}=T_1$. Assume otherwise. Then $(s_{\mathrm{min}}+h/2)\wedge T_2\in S$ implies $[T_1\vee(s_{\mathrm{min}}-h/2),(s_{\mathrm{min}}+h/2)\wedge T_2]\subseteq S$, which in turn leads to
$[T_1\vee(s_{\mathrm{min}}-h/2),T_2]\subseteq S$ contradicting the definition of $s_{\mathrm{min}}$. Thus $s_{\mathrm{min}}=T_1$. In particular $(T_1+h)\wedge T_2\in S$, which implies $T_1\in S$.
\end{proof}

Similarly we can show:

\begin{lem}[Small interval induction, forward]
\label{siif}
Let $T_1<T_2$ be real numbers and let $S\subseteq [T_1,T_2]$ s.t.
\begin{itemize}
\item $T_1\in S$,
\item there exists an $h>0$ s.t. $[s,s+h]\cap [T_1,T_2]\subseteq S$ for all $s\in S$.
\end{itemize}
Then $S=[T_1,T_2]$. In particular $T_2\in S$.
\end{lem}

Here is a first application of this technique.

\begin{cor}[Global uniqueness]\label{uniq}
Let $\mu,\sigma,f,\xi$ be as in Theorem \ref{locexist}. Assume that there are two weakly regular decoupling fields $u^{(1)},u^{(2)}$ to the corresponding problem on some interval $[t,T]$. Then $u^{(1)}=u^{(2)}$ up to modifications.
\end{cor}
\begin{proof}
Let $S\subseteq [t,T]$ be the set of all times $s$, s.t. $u^{(1)}(s,\cdot)=u^{(2)}(s,\cdot)$ a.e.
\begin{itemize}
\item Obviously $T\in S$.
\item Let $s\in S$ be arbitrary. According to Theorem \ref{locexist} there exists an $h>0$ such that there is a unique decoupling field $\tilde{u}$ to $\fbsde (u^{(1)}(s,\cdot),(\mu,\sigma,f))=\fbsde (u^{(2)}(s,\cdot),(\mu,\sigma,f))$ on the interval $[(s-h)\vee t,s]$ s.t. $L_{\tilde{u},x}<L_{\sigma,z}^{-1}$, $\|\tilde{u}(\cdot,\cdot,0)\|_{\infty}<\infty$, and where $h$ can be chosen independent of $s$. This means that $\tilde{u},u^{(1)},u^{(2)}$ coincide on $[(s-h)\vee t,s]$, hence $[(s-h)\vee t,s]\subseteq S$.
\end{itemize}
This shows $S=[t,T]$ by small interval induction.
\end{proof}

After having shown uniqueness of $u$ we show its strong regularity.

\begin{cor}[Global regularity]\label{regul}
Let $\mu,\sigma,f,\xi$ be as in Theorem \ref{locexist}. Assume that there exists a weakly regular decoupling field $u$ to this problem on some interval $[t,T]$.Then $u$ is strongly regular.
\end{cor}
\begin{proof}
Let $S\subseteq [t,T]$ be the set of all times $s\in[t,T]$ s.t. $u$ is strongly regular on $[t,s]$.
\begin{itemize}
\item Obviously $t\in S$, since we can choose $Z=0$, $X=x$, $Y=u(t,\cdot)$.
\item Let $s\in S$ be arbitrary. According to Lemma \ref{gluestrongly regular} there exists an $h>0$ s.t. $u$ is strongly regular on $[t,(s+h)\wedge T]$ since $L_{u((s+h)\wedge T,\cdot)}<\infty$ and $L_{\sigma,z}L_{u((s+h)\wedge T,\cdot)}<1$. Recalling Remark \ref{hchoice} and the requirements $L_{u,x}<\infty$, $L_{\sigma,z}L_{u,x}<1$, we can choose $h$ independent of $s$.
\end{itemize}
This shows $S=[t,T]$ by small interval induction.
\end{proof}

Notice that Corollary \ref{uniq} only provides uniqueness of weakly regular decoupling fields, not uniqueness of processes $(X,Y,Z)$ solving the FBSDE in the classical sense. However, we can show:

\begin{cor}\label{uniqXYZ}
Let $\mu,\sigma,f,\xi$ be as in Theorem \ref{locexist}. Assume that there exists a weakly regular decoupling field $u$ of the corresponding FBSDE on some interval $[t,T]$.\\ Then for any initial condition $X_t=x\in\mathbb{R}^n$ there is a unique $\mathbb{G}_0$ - solution $X,Y,Z$ of the FBSDE on $[t,T]$.
\end{cor}
\begin{proof}
The existence of $(X,Y,Z)$ follows directly from strong regularity of $u$ (Corollary \ref{regul}). In fact we even have $(X,Y,Z)\in \mathbb{G}_{t}$. Let us show uniqueness. Assume we have some $\mathbb{G}_0$ - solution $(X,Y,Z)$. Due to strong regularity, we only need to show the decoupling condition $Y_s=u(s,X_s)$. Choose an appropriate $t'\in[t,T]$ according to Remark \ref{hchoice}. Then we have $Y_s=u(s,X_s)$ for $s\in[t',T]$ according to Remark \ref{xyz}. In particular we have $Y_{t'}=u(t',X_{t'})$, which serves as a new terminal condition. We can now repeat the argument going to the left and conclude the proof using small interval induction (backward).
\end{proof}

Now we want to explore how large the interval $[t,T]$ can be chosen, such that we still have (weakly regular) decoupling fields on this interval. It is natural to work with the following definition.

\begin{defi}
We define the \emph{maximal interval} $I_{\mathrm{max}}\subseteq[0,T]$ for $\fbsde(\xi,(\mu,\sigma,f))$ as the union of all intervals $[t,T]\subseteq[0,T]$, such that there exists a weakly regular decoupling field $u$ on $[t,T]$.
\end{defi}

Unfortunately the maximal interval might very well be open to the left. Therefore we need to make our notions more precise in the following definitions.
\begin{defi}
Let $t<T$. We call a function $u:(t,T]\times\mathbb{R}^n\rightarrow\mathbb{R}^m$ a decoupling field for $\fbsde(\xi,(\mu,\sigma,f))$ on $(t,T]$ if $u$ restricted to $[t',T]$ is a decoupling field for all $t'\in(t,T]$.
\end{defi}

\begin{defi}
Let $t<T$. We call a decoupling field $u$ on $(t,T]$ \emph{weakly regular} if $u$ restricted to $[t',T]$ is weakly regular for all $t'\in(t,T]$.
\end{defi}

\begin{defi}
Let $t<T$. We call a decoupling field $u$ on $(t,T]$ \emph{strongly regular} if $u$ restricted to $[t',T]$ is strongly regular for all $t'\in(t,T]$.
\end{defi}


Now we can show the main result of this paper.

\begin{thm}[Global existence in weak form]\label{globalexist}
Let $\mu,\sigma,f,\xi$ be as in Theorem \ref{locexist}. Then there exists a unique weakly regular decoupling field $u$ on $I_{\mathrm{max}}$. This $u$ is even strongly regular. \\
Furthermore either $I_{\mathrm{max}}=[0,T]$ or $I_{\mathrm{max}}=(t_{\mathrm{min}},T]$ where $0\leq t_{\mathrm{min}}<T$.
\end{thm}

\begin{proof}
Let $t\in I_{\mathrm{max}}$. Obviously there exists a decoupling field $\check{u}^{(t)}$ on $[t,T]$ satisfying $L_{\check{u}^{(t)},x}<L_{\sigma,z}^{-1}$ and $\|\check{u}(\cdot,\cdot,0)\|_{\infty}<\infty$. There is only one such $\check{u}^{(t)}$ by Corollary \ref{uniq}. Furthermore, for $t,t'\in I_{\mathrm{max}}$ the functions $\check{u}^{(t)}$ and $\check{u}^{(t')}$ coincide on $[t\vee t',T]$ according to Corollary \ref{uniq}. \\
Define $u(t,\cdot):=\check{u}^{(t)}(t,\cdot)$ for all $t\in I_{\mathrm{max}}$. This function $u$ is a decoupling field on $[t,T]$ since it coincides with $\check{u}^{(t)}$ on $[t,T]$. Therefore $u$ is a decoupling field on the whole interval $I_{\mathrm{max}}$ and satisfies $L_{u|_{[t,T]},x}<L_{\sigma,z}^{-1}$ for all $t\in I_{\mathrm{max}}$. \\
Uniqueness of $u$ follows directly from Corollary \ref{uniq} applied to every interval $[t,T]\subseteq I_{\mathrm{max}}$. \\
Furthermore, $u$ is strongly regular on $[t,T]$ for all $t\in I_{\mathrm{max}}$ because of Corollary \ref{regul}. \\
To prove the claim on the form of $I_{\mathrm{max}}$, note that $I_{\mathrm{max}}=[t,T]$ with $t\in(0,T]$ is not possible. Assume otherwise. According to Theorem \ref{globalexist} there is a decoupling field $u$ on $[t,T]$ s.t. $L_{u,x}<L_{\sigma,z}^{-1}$ and $\|u(\cdot,\cdot,0)\|_{\infty}<\infty$. However, then we can extend $u$ a little bit to the left using Theorem \ref{locexist} and Lemma \ref{glue}.
\end{proof}

By \emph{global existence in strong form} we mean the above weak global existence together with $I_{\mathrm{max}}=[0,T]$. Unfortunately the "bad" case $I_{\mathrm{max}}=(t_{\mathrm{min}},T]$ is possible and is even more common. The following result basically says that this case can only occur if there is an "explosion" in the spatial derivative of $u$ as we approach the lower boundary $t_{\mathrm{min}}$. By "explosion" we mean reaching the "forbidden" value $L_{\sigma,z}^{-1}$ which is just $\infty$ in many applications.

\begin{lem}\label{explosion}
Let $\mu,\sigma,f,\xi$ be as in Theorem \ref{locexist}. If $I_{\mathrm{max}}=(t_{\mathrm{min}},T]$, then
$$\lim_{t\downarrow t_{\mathrm{min}}}L_{u(t,\cdot),x}=L_{\sigma,z}^{-1},$$
where $u$ is the unique decoupling field on $I_{\mathrm{max}}$.
\end{lem}
\begin{proof}
This can be shown by contradiction. Assume in fact that we can select times $t_n\downarrow t_{\mathrm{min}}$ as $n\rightarrow\infty$ such that $$\sup_{n\in\mathbb{N}}L_{u(t_n,\cdot),x}<L_{\sigma,z}^{-1}.$$
Then we can choose $h>0$ according to Remark \ref{hchoice}, which does not depend on $n$ and then choose $n$ large enough to have $t_n-t_{\mathrm{min}}<h$. Hence, $u$ can be extended to the left (using Lemma \ref{glue}) to a larger interval $[(t_n-h)\vee 0,T]$ contradicting the definition of $I_{\mathrm{max}}$.
\end{proof}

Lemma \ref{explosion} serves as a blueprint to show strong global existence in those cases in which it is suspected to hold. Let us describe the different steps.

\begin{enumerate}
\item Assume indirectly that $I_{\mathrm{max}}=[0,T]$ does not hold, which implies $I_{\mathrm{max}}=(t_{\mathrm{min}},T]$. Choose arbitrary $t\in (t_{\mathrm{min}},T]$, $x\in\mathbb{R}^n$ and consider the corresponding FBSDE.
\item Differentiate the FBSDE w.r.t. $x$. This is possible because of strong regularity of $u$ (Theorem \ref{globalexist}). We obtain joint dynamics of $\frac{\dx }{\dx x}X,\frac{\dx }{\dx x}Y,\frac{\dx }{\dx x}Z$.
\item Using It\^o's formula deduce the dynamics of $\frac{\dx }{\dx x}Y_s(\frac{\dx }{\dx x}X_s)^{-1}$. Note that this process is equal to $u_x(s,X_s)$, as a consequence of the decoupling condition $Y_s=u(s,X_s)$.
\item Using the dynamics of $u_x(s,X_s)$ show that its modulus can be bounded away from $L_{\sigma,z}^{-1}$ independently of $t,x,s,\omega$. This contradicts Lemma \ref{explosion} and therefore $I_{\mathrm{max}}=[0,T]$ must hold.
\end{enumerate}

This blueprint can be referred to as the \emph{method of decoupling fields} to show global existence of solutions to FBSDEs (note Corollary \ref{uniqXYZ} at this point). We have succesfully applied this method to a number of strongly coupled FBSDEs ranging from problems appearing in utility maximization to Skorohod embedding. This will be discussed in detail in forthcoming publications.

\subsection{The Markovian case}
A problem given by $\mu,\sigma,f,\xi$ is said to be \emph{Markovian}, if these four functions are deterministic, i.e. depend on $t,x,y,z$ only. We will call these functions deterministic and the case Markovian also if the dependence on $\omega$ is trivial, i.e. the values for $(\omega,t,x,y,z)$ and $(\omega',t,x,y,z)$ are the same for all $\omega,\omega'\in\tilde{\Omega}$ and all $t,x,y,z$ where $\tilde{\Omega}\subseteq\Omega$ is measurable and has probability $1$.\\
In the Markovian case we can somewhat relax the Lipschitz continuity assumption and still obtain local existence together with uniqueness. What makes the Markovian case so special is the property
$$"Z_s=u_x(s,X_s)\cdot\sigma(s,X_s,Y_s,Z_s)"$$
which comes from the fact that $u$ will also be deterministic. This property allows us to bound $Z$ by a constant if we assume that $\sigma$ is bounded. This is a common trick (e.g. \cite{richou}, \cite{richou2}). \\
This relationship can be seen as a consequence of the It\^o formula, applied to $u(s,X_s)$, assuming that $u$ is smooth enough. However, under our assumptions $u$ will not have sufficient smoothness.\\
In the literature sometimes Malliavin's calculus is used to deduce such a relationship (e.g. \cite{dreis}). However we will follow simpler arguments. Note that under our assumptions $u$ is only weakly differentiable in $x$, such that $u_x(s,\cdot)$ is only unique up to null sets. At the same time the distribution of $X_s$ does not have to be absolutely continuous w.r.t. the Lebesgue measure on $\mathbb{R}^n$. Therefore, $u_x(s,X_s)$ is not properly defined and so we will not actually try to show $Z_s=u_x(s,X_s)\cdot\sigma(s,X_s,Y_s,Z_s)$. We are only interested in bounding $Z$. \\
Finally let us remark that in the Markovian case the FBSDEs we consider are closely connected with a major class of quasilinear parabolic partial differential equations (via Feynman-Kac). A decoupling field $u$ can be seen as a type of solution to such a PDE. Notice that we are able to develop an existence and uniqueness theory under very mild assumptions. Basically we only need Lipschitz or even just local Lipschitz continuity (as we will see later) of the parameters involved. Under our assumptions $u$ will only be Lipschitz continuous in space and continuous as a function of time and space, but in order to write down the classical PDE $u$ has to be differentiable in time (at least weakly) and twice differentiable in space, which requires much more restrictive assumptions. Thus we have a very weak form of solvability allowing us to have existence and uniqueness for a very general class of problems. \\
Also note that with this stochastic interpretation of second order PDEs we are able to implement a rather explicit construction. So far we have either used Picard iterations and Banach's fixed point theorem or a technique of "gluing together" decoupling fields on adjacent intervals. This gives us more control over the objects constructed and can potentially serve as basis for numerical methods. \\

First let us prove the following statement.

\begin{lem}\label{deter}
 Let $\mu,\sigma,f,\xi$ be deterministic. Assume we have a unique strongly regular decoupling field $u$ on an interval $[t,T]$. Then $u$ is also deterministic.
\end{lem}
\begin{proof}
We can decompose $\Omega=\Omega_{[0,t]}\times\Omega_{(t,T]}$, where the first component contains only the information about $\mathcal{F}_0$ and the noise on $[0,t]$, while the second component only contains the information about the noise on $[t,T]$, i.e. is generated by $(W_{t+h}-W_t)_{h\in [0,T-t]}$. We fix a suitable $\omega_1\in \Omega_{[0,t]}$ and consider $\tilde{u}:=u((\omega_1,\cdot),\cdot,\cdot)$. It is straightforward to see that this function will also be a strongly regular decoupling field since $\mu,\sigma,f,\xi$ are deterministic and $(W_{t+h}-W_t)_{h\in [0,T-t]}$ does not depend on $\omega_1$. Due to uniqueness $\tilde{u}$ must coincide a.e. with $u$. Since $\tilde{u}(t,x)=Y^{x}_t$ is $\mathcal{F}_t$-measurable, it is independent of $(W_{t+h}-W_t)_{h\in [0,T-t]}$. Thus $\tilde{u}(t,x)$ and therefore $u(t,x)$ can be assumed deterministic.
\end{proof}

As an application we show the following very fundamental result.

\begin{lem}\label{Zcontrol}
Let $\mu,\sigma,f,\xi$ be as in Theorem \ref{locexist} and assume in addition that they are deterministic. Let $u$ be a strongly regular decoupling field on an interval $[t,T]$. Choose $t_1<t_2$ from $[t,T]$ and an initial condition $X_{t_1}$. Then the corresponding $Z$ satisfies $\|Z\|_\infty\leq L_{u,x}\cdot\|\sigma\|_\infty$.
\end{lem}
\begin{proof}\textsc{Case 1:} Assume $X_{t_1}=x\in\mathbb{R}^n$.\\ $u$ is unique according to Corollary \ref{uniq}. It is also deterministic according to Lemma \ref{deter}.\\
Notice $\lim_{s\downarrow s'}\frac{1}{s-s'}\int_{s'}^s Z_r(\omega)\dx r=Z_{s'}(\omega)$ for almost all $(\omega,s')\in\Omega\times[t,T]$ due to the fundamental Theorem of Lebesgue integral calculus. The same holds for the expressions $\frac{1}{s-s'}\int_{s'}^s W_r f(r,X_r,Y_r,Z_r)\dx r$ and $\frac{1}{s-s'}\int_{s'}^s f(r,X_r,Y_r,Z_r)\dx r$.\\
Also $\mathbb{E}[|Z_{s'}|^2]<\infty$ for almost all $s'\in[t,T]$ due to $\mathbb{E}\left[\int_{t}^T |Z_r|^2\dx r\right]<\infty$. The same holds for the expression $\mathbb{E}[|f(s',X_{s'},Y_{s'},Z_{s'})|^2]$. \\
Choose an $s'\in [t,T]$ with
\begin{itemize}
\item $\lim_{s\downarrow s'}\frac{1}{s-s'}\int_{s'}^s Z_r\dx r=Z_{s'}$ a.s.,
\item $\lim_{s\downarrow s'}\frac{1}{s-s'}\int_{s'}^s W_r f(r,X_r,Y_r,Z_r)\dx r$ a.s.,
\item $\lim_{s\downarrow s'}\frac{1}{s-s'}\int_{s'}^s f(r,X_r,Y_r,Z_r)\dx r$ a.s.,
\item $\mathbb{E}[|Z_{s'}|^2],\mathbb{E}[|f(s',X_{s'},Y_{s'},Z_{s'})|^2]<\infty$.
\end{itemize}
For every $(\mathcal{F}_s)_{s\in [s',T]}$ - stopping time $\tau:\Omega\rightarrow (s',T]$ we have according to the product rule
\begin{multline*}
 Y_\tau(W_\tau-W_{s'})=\\
=\int_{s'}^\tau Y_r\dx W_r+\int_{s'}^\tau (W_r-W_{s'})f(r,X_r,Y_r,Z_r)\dx r+\int_{s'}^\tau Z_r\dx W_r  (W_r-W_{s'})+\int_{s'}^\tau Z_r\dx r.
\end{multline*}
$\tau$ can be chosen in such a way that
\begin{itemize}
\item $\left(\int_{s'}^{\tau\wedge s} Y_r\dx W_r\right)_{s\in[s',T]}$ and $\left(\int_{s'}^{\tau\wedge s} Z_r \dx W_r (W_r-W_{s'})\right)_{s\in[s',T]}$ are uniformly integrable martingales,
\item $\left|\frac{1}{s-s'}\int_{s'}^{\tau\wedge s} Z_r\dx r\right|\leq |Z_{s'}|+1$ a.s.,
\item $\left|\frac{1}{s-s'}\int_{s'}^{\tau\wedge s} f(r,X_r,Y_r,Z_r)\dx r\right|\leq |f(s',X_{s'},Y_{s'},Z_{s'})|+1$ a.s.,
\item $\left|\frac{1}{s-s'}\int_{s'}^{\tau\wedge s} W_r f(r,X_r,Y_r,Z_r)\dx r\right|\leq |W_{s'}f(s',X_{s'},Y_{s'},Z_{s'})|+1$ a.s., for all $s\in[s',T]$.
\end{itemize}
We only discuss the second statement: Define $U_s:=\frac{1}{s-s'}\int_{s'}^{s} Z_r\dx r$ for $s\in(s',T]$ and set $U_{s'}:=Z_{s'}$. $U$ is a continuous and adapted process starting at $Z_{s'}$. If we choose $\tau$ such that the stopping occurs when $|U|$ reaches $|Z_{s'}|+1$, then $\tau>s'$ will hold and also $\left|\frac{1}{s-s'}\int_{s'}^{\tau\wedge s} Z_r\dx r\right|=\left|\frac{\tau\wedge s-s'}{s-s'}U_{\tau\wedge s}\right|\leq \left|U_{\tau\wedge s}\right|\leq |Z_{s'}|+1$, for $s\in(s',T]$. If we stop earlier, i.e. choose a smaller stopping time, the bound still holds as long as $\tau>s'$.

Now we take conditional expectations and pass to the limit applying dominated convergence along the way. This leads to
\begin{multline*}
 \lim_{s\downarrow s'}\mathbb{E}\left[\frac{1}{s-s'}Y_{\tau\wedge s}(W_{\tau\wedge s}-W_{s'})\bigg|\mathcal{F}_{s'}\right]=\\
=\lim_{s\downarrow s'}\mathbb{E}\left[\frac{1}{s-s'}\int_{s'}^{\tau\wedge s} (W_r-W_{s'})f(r,X_r,Y_r,Z_r)\dx r\bigg|\mathcal{F}_{s'} \right]+
\lim_{s\downarrow s'}\mathbb{E}\left[\frac{1}{s-s'}\int_{s'}^{\tau\wedge s} Z_r\dx r\bigg|\mathcal{F}_{s'} \right]=\\
=\mathbb{E}\left[\lim_{s\downarrow s'}\frac{1}{s-s'}\int_{s'}^{\tau\wedge s} W_rf(r,X_r,Y_r,Z_r)\dx r\bigg|\mathcal{F}_{s'} \right]- \\
-\mathbb{E}\left[\lim_{s\downarrow s'}\frac{1}{s-s'}\int_{s'}^{\tau\wedge s} W_{s'}f(r,X_r,Y_r,Z_r)\dx r\bigg|\mathcal{F}_{s'} \right]
+\mathbb{E}\left[\lim_{s\downarrow s'}\frac{1}{s-s'}\int_{s'}^{\tau\wedge s} Z_r\dx r\bigg|\mathcal{F}_{s'} \right]=\\
=W_{s'}f({s'},X_{s'},Y_{s'},Z_{s'})-W_{s'}f({s'},X_{s'},Y_{s'},Z_{s'})+Z_{s'}=Z_{s'},
\end{multline*}
where we used $\tau(\omega)\wedge s=s$ for $s\in(s',T]$ small enough. \\
This of course implies
$$  \lim_{s\downarrow s'}\left|\mathbb{E}\left[\frac{1}{s-s'}Y_{\tau\wedge s}(W_{\tau\wedge s}-W_{s'})\bigg|\mathcal{F}_{s'}\right]\right|=|Z_{s'}|. $$
At the same time
\begin{multline*}
 \left|\mathbb{E}\left[\frac{1}{s-s'}Y_{\tau\wedge s}(W_{\tau\wedge s}-W_{s'})\bigg|\mathcal{F}_{s'}\right]\right|\leq \\
\leq\left|\mathbb{E}\left[\frac{1}{s-s'}Y_{s}(W_{\tau\wedge s}-W_{s'})\bigg|\mathcal{F}_{s'}\right]\right|+\left|\mathbb{E}\left[\frac{1}{s-s'}\left(Y_{s}-Y_{\tau\wedge s}\right)(W_{\tau\wedge s}-W_{s'})\bigg|\mathcal{F}_{s'}\right]\right|=\\
=\left|\mathbb{E}\left[\frac{1}{s-s'}u(s,X_s)(W_{\tau\wedge s}-W_{s'})\bigg|\mathcal{F}_{s'}\right]\right|+\left|\mathbb{E}\left[\frac{1}{s-s'}\mathbb{E}\left[Y_{s}-Y_{\tau\wedge s}|\mathcal{F}_{\tau\wedge s}\right](W_{\tau\wedge s}-W_{s'})\bigg|\mathcal{F}_{s'}\right]\right|=
\end{multline*}
\begin{multline*}
=\left|\mathbb{E}\left[\frac{1}{s-s'}\left(u(s,X_s)-u(s,X_{s'})\right)(W_{\tau\wedge s}-W_{s'})\bigg|\mathcal{F}_{s'}\right]\right|+\\
\shoveright{+\left|\mathbb{E}\left[\frac{1}{s-s'}\mathbb{E}\left[\int_{\tau\wedge s}^s f(r,X_r,Y_r,Z_r)\dx r \bigg|\mathcal{F}_{\tau\wedge s}\right](W_{\tau\wedge s}-W_{s'})\bigg|\mathcal{F}_{s'}\right]\right|\leq} \\
\shoveleft{\leq\frac{1}{s-s'}\left(\mathbb{E}\left[\left|u(s,X_s)-u(s,X_{s'})\right|^2\big|\mathcal{F}_{s'}\right]\right)^{\frac{1}{2}}\left(\mathbb{E}\left[(W_{\tau\wedge s}-W_{s'})^2\big|\mathcal{F}_{s'}\right]\right)^{\frac{1}{2}}+}\\
+\frac{1}{s-s'}\left(\mathbb{E}\left[\left|\int_{\tau\wedge s}^s f(r,X_r,Y_r,Z_r)\dx r\right|^2 \bigg|\mathcal{F}_{s'}\right]\right)^{\frac{1}{2}}\left(\mathbb{E}\left[(W_{\tau\wedge s}-W_{s'})^2\big|\mathcal{F}_{s'}\right]\right)^{\frac{1}{2}}\leq
\end{multline*}
\begin{multline*}
\leq\frac{1}{s-s'}L_{u,x}\left(\mathbb{E}\left[\left|X_s-X_{s'}\right|^2\big|\mathcal{F}_{s'}\right]\right)^{\frac{1}{2}}\sqrt{s-s'}+\\
+\frac{1}{s-s'}\left(\mathbb{E}\left[(s-s')\int_{s'}^s |f(r,X_r,Y_r,Z_r)|^2\dx r \bigg|\mathcal{F}_{s'}\right]\right)^{\frac{1}{2}}\sqrt{s-s'}\leq
\end{multline*}
\begin{multline*}
\leq\frac{1}{\sqrt{s-s'}}L_{u,x}\left(\left(\mathbb{E}\left[\left|\int_{s'}^s \mu(r,X_r,Y_r,Z_r)\dx r\right|^2\bigg|\mathcal{F}_{s'}\right]\right)^{\frac{1}{2}}+
\left(\mathbb{E}\left[\int_{s'}^s |\sigma(r,X_r,Y_r,Z_r)|^2\dx r\bigg|\mathcal{F}_{s'}\right]\right)^{\frac{1}{2}}\right)+\\
\shoveright{+\left(\mathbb{E}\left[\int_{s'}^s |f(r,X_r,Y_r,Z_r)|^2\dx r \bigg|\mathcal{F}_{s'}\right]\right)^{\frac{1}{2}}\leq}\\
\leq L_{u,x}\left(\mathbb{E}\left[\int_{s'}^s |\mu(r,X_r,Y_r,Z_r)|^2\dx r\bigg|\mathcal{F}_{s'}\right]\right)^{\frac{1}{2}}+
L_{u,x}\|\sigma\|_{\infty}+
\left(\mathbb{E}\left[\int_{s'}^s |f(r,X_r,Y_r,Z_r)|^2\dx r \bigg|\mathcal{F}_{s'}\right]\right)^{\frac{1}{2}},
\end{multline*}
which a.s. converges to $L_{u,x}\|\sigma\|_{\infty}$ as $s\rightarrow s'$. \\
We have therefore demonstrated $|Z_{s'}|\leq L_{u,x}\|\sigma\|_{\infty}$ a.s.\\
Note that this argument works for almost all $s'\in[t,T]$, as mentioned in the beginning.

\textsc{Case 2:} We can decompose $\Omega=\Omega_{[0,t]}\times\Omega_{(t,T]}$, where the first component contains only the information about $\mathcal{F}_0$ and the noise on $[0,t]$, while the second component is generated by $(W_{t+h}-W_t)_{h\in [0,T-t]}$. If we fix the first component $\omega_1\in \Omega_{[0,t]}$, $X_{t_1}(\omega_1,\cdot)$ becomes a constant from $\mathbb{R}^n$ and $X(\omega_1,\cdot),Y(\omega_1,\cdot),Z(\omega_1,\cdot)$ still solve the FBSDE together with the decoupling condition. This implies $\|Z(\omega_1,\cdot)\|_\infty\leq L_{u,x}\cdot\|\sigma\|_\infty$. Since $\omega_1$ is arbitrary, we have $\|Z\|_\infty\leq L_{u,x}\cdot\|\sigma\|_\infty$.
\end{proof}

Next we investigate the continuity of $u$ as a function of time and space.

\begin{lem}\label{contin}
Assume that $\mu,\sigma,f$ satisfy
$$(|\mu|^2+|\sigma|^2+|f|^2)(t,x,y,z)\leq C_1+C_2|x|^2$$
with some fixed constants $C_1,C_2>0$ (and arbitrary $t,x,y,z$). \\
If we have a strongly regular and deterministic decoupling field $u$ to $\fbsde(\xi,(\mu,\sigma,f))$ on an interval $[t,T]$, then $u$ is continuous.
\end{lem}
\begin{proof}
Choose any $t_1<t_2$ from the interval $[t,T]$ and some $x\in\mathbb{R}^n$ as initial value. Then
$$ X_s=x+\int_{t_1}^s\mu(r,X_r,Y_r,Z_r)\dx r+\int_{t_1}^s\sigma(r,X_r,Y_r,Z_r)\dx W_r, $$
$$ Y_s=Y_{t_2}-\int_{s}^{t_2}f(r,X_r,Y_r,Z_r)\dx r-\int_{s}^{t_2}Z_r^\top\dx W_r, $$
$$ Y_s=u(s,X_s). $$
First of all note that the required inequality allows us to control $\mathbb{E}[|X_r|^2]$ independently of the interval $[t_1,t_2]$ or $r\in[t_1,t_2]$. This is because
$$ |X_s|^2=x^2+2\int_{t_1}^sX_r^\top\mu(r,X_r,Y_r,Z_r)\dx r+2\int_{t_1}^sX_r^\top\sigma(r,X_r,Y_r,Z_r)\dx W_r+\int_{t_1}^s|\sigma|^2(r,X_r,Y_r,Z_r)\dx r, $$
due to It\^o's formula. So
$$ \mathbb{E}[|X_s|^2]\leq x^2+\int_{t_1}^s\left( C+C\mathbb{E}[|X_r|^2]\right)\dx r $$
with some constant $C>0$, depending on $C_1,C_2$. Using Gronwall's lemma we obtain
\begin{equation}\label{Xcontrol}
\mathbb{E}[|X_s|^2]\leq (x^2+(T-t)C)e^{C(T-t)}.
\end{equation}

We already know that $u$ is Lipschitz continuous in $x$. It remains to show the following H\"older continuity property:
\begin{equation}\label{ucontrol}
|u(t_1,x)-u(t_2,x)|\leq (C+C|x|)|t_1-t_2|^{\frac{1}{2}}
\end{equation}
 for some constant $C>0$ and all $t_1, t_2\in[t,T]$, $x\in\mathbb{R}^n$.

We use the triangle inequality together with the decoupling condition $Y_s=u(s,X_s)$, to get
\begin{multline*}
|u(t_2,x)-u(t_1,x)|\leq |u(t_2,x)-\mathbb{E}[u(t_2,X_{t_2})]|+|\mathbb{E}[u(t_2,X_{t_2})]-u(t_1,x)|= \\
=|\mathbb{E}[u(t_2,x)-u(t_2,X_{t_2})]|+|\mathbb{E}[Y_{t_2}-Y_{t_1}]|\leq L\mathbb{E}[|X_{t_2}-x|]+|\mathbb{E}[Y_{t_2}-Y_{t_1}]|.
\end{multline*}
The Cauchy-Schwarz and Minkowski inequalities lead to
$$ |u(t_2,x)-u(t_1,x)|\leq  L\sqrt{\mathbb{E}[|X_{t_2}-x|^2]}+|\mathbb{E}[Y_{t_2}-Y_{t_1}]| \leq $$
$$ \leq L\left(\mathbb{E}\left[\left|\int_{t_1}^{t_2}\mu(r,X_r,Y_r,Z_r)\dx r\right|^2\right]\right)^{\frac{1}{2}}+L\left(\mathbb{E}\left[\left|\int_{t_1}^{t_2}\sigma(r,X_r,Y_r,Z_r)\dx W_r\right|^2\right]\right)^{\frac{1}{2}}+ $$
$$+\left|\int_{t_1}^{t_2}\mathbb{E}[f(r,X_r,Y_r,Z_r)]\dx r\right|\leq $$
$$ \leq L\sqrt{t_2-t_1}\left(\mathbb{E}\left[\int_{t_1}^{t_2}|\mu(r,X_r,Y_r,Z_r)|^2\dx r\right]\right)^{\frac{1}{2}}+L\left(\mathbb{E}\left[\int_{t_1}^{t_2}|\sigma|^2(r,X_r,Y_r,Z_r)\dx r\right]\right)^{\frac{1}{2}}+ $$
$$+\int_{t_1}^{t_2}\mathbb{E}[\left|f(r,X_r,Y_r,Z_r)\right|]\dx r\leq $$
$$ \leq L\sqrt{t_2-t_1}\left(\mathbb{E}\left[\int_{t_1}^{t_2}C_1+C_2|X_r|^2\dx r\right]\right)^{\frac{1}{2}}+L\left(\mathbb{E}\left[\int_{t_1}^{t_2}C_1+C_2|X_r|^2\dx r\right]\right)^{\frac{1}{2}}+ $$
$$+\int_{t_1}^{t_2}\sqrt{\mathbb{E}[C_1+C_2|X_r|^2]}\dx r\leq $$
$$ \leq L\left(C+C\sqrt{\sup_{r\in[t_1,t_2]}\mathbb{E}[|X_r|^2]}\right)\sqrt{t_2-t_1} \sqrt{t_2-t_1}+L\left(C+C\sqrt{\sup_{r\in[t_1,t_2]}\mathbb{E}[|X_r|^2]}\right)\sqrt{t_2-t_1}+$$
$$+\left(C+C\sqrt{\sup_{r\in[t_1,t_2]}\mathbb{E}[|X_r|^2]}\right)(t_2-t_1), $$
with some constant $C$ depending on $C_1,C_2$. Now using (\ref{Xcontrol}) we obtain inequality (\ref{ucontrol}).
\end{proof}

Now we come to a local existence result.

\begin{thm}\label{LocLip}
Let
\begin{itemize}
\item $\mu,\sigma,f$ be
\begin{itemize}
\item deterministic,
\item Lipschitz continuous in $x,y,z$ on sets of the form $[0,T]\times\mathbb{R}^n\times B_1 \times B_2$, where $B_1\subset \mathbb{R}^{m}$ and $B_2\subset \mathbb{R}^{m\times d}$ are arbitrary bounded sets
\item and such that $\|\mu(\cdot,0,0,0)\|_\infty,\|f(\cdot,\cdot,0,0)\|_{\infty},\|\sigma\|_{\infty},\|\sigma_z\|_{\infty}<\infty$,
\end{itemize}
\item $\xi: \mathbb{R}^n\rightarrow \mathbb{R}^m$ be bounded and such that $L_{\xi,x}<L_{\sigma,z}^{-1}$.
\end{itemize}
Then there exists a time $t\in[0,T)$ such that $\fbsde (\xi,(\mu,\sigma,f))$ has a unique bounded and weakly regular decoupling field $u$ on $[t,T]$. This $u$ is also
\begin{itemize}
\item strongly regular,
\item deterministic,
\item continuous and
\item satisfies $\sup_{t_1,t_2,X_{t_1}}\left(\|Y\|_\infty+\|Z\|_\infty\right)<\infty$, where $t_1<t_2$ are from $[t,T]$ and $X_{t_1}$ is an initial value (see the definition of a decoupling field for the meaning of these variables).
\end{itemize}
\end{thm}
\begin{proof}
For any constant $H>0$ let $\chi_{H}:\mathbb{R}^m\times\mathbb{R}^{m\times d}\rightarrow \mathbb{R}^m\times\mathbb{R}^{m\times d}$ be the projection onto the ball of radius $H$ with center $0\in \mathbb{R}^m\times\mathbb{R}^{m\times d}$. Note that $\chi_{H}$ is Lipschitz continuous with Liptschitz constant $L_{\chi_{H}}=1$ and bounded such that $\|\chi_{H}\|_{\infty}=H$.

We implement an "inner cutoff" by defining $\mu_H,\sigma_H,f_H$ via $\mu_H(t,x,y,z):=\mu(t,x,\chi_H(y,z))$, etc.

This makes $\mu_H,\sigma_H,f_H$ Lipschitz continuous with some Lipschitz constant $L_H$. Furthermore $L_{\sigma_H,z}\leq L_{\sigma,z}$. According to Theorem \ref{locexist} we know that the problem given by $\fbsde (\xi,(\mu_H,\sigma_H,f_H))$ has a unique solution $u$ with $L_{u,x}<L_{\sigma,z}^{-1}$ and $\|u(\cdot,\cdot,0)\|_{\infty}<\infty$ on some small interval. We also know that this $u$ is strongly regular. Furthermore, $u$ is deterministic (Lemma \ref{deter}) and continuous (Lemma \ref{contin}).

We will show that for sufficiently large $t<T$ it will also be a solution to the problem $\fbsde (\xi,(\mu,\sigma,f))$. \\

Using Remark \ref{bbbbrem}, we can bound the operator norm $\sup_{v\in S^{n-1}}|u_x|_v$ of $u_x$ by $L_{\xi,x}+C_H(T-t)^{\frac{1}{4}},$
where $C_H<\infty$ is some constant, which does not depend on $t\in[t',T]$, where $t'<T$ is fixed, but does depend on $H$. Now using Lemma \ref{Zcontrol} we have
$$ \|Z\|_\infty\leq C_\sigma L_{\xi,x}+C_\sigma C_H(T-t)^{\frac{1}{4}}, $$
where $C_\sigma$ is proportional to $\|\sigma\|_\infty$. \\

We can bound $u$ itself as well. In fact, using the boundedness of $\xi$ and
$$Y_s+\int_{s}^{T}Z_r\dx W_r=\xi(X_T)-\int_{s}^{T}f_H(r,X_r,Y_r,Z_r)\dx r,\qquad s\in [t,T],$$
we conclude that $Y$ must be bounded, due to Lipschitz continuity of $f$ on compact sets and the use of a cutoff (the boundedness of $f(r,X_r,0,0)$ is also used):
$$ \left(|Y_s|^2+\mathbb{E}\left[\int_{s}^{T}|Z_r|^2\dx r\bigg|\mathcal{F}_s\right]\right)^{\frac{1}{2}}
\leq \|\xi\|_\infty+\sqrt{T-s}\left(\mathbb{E}\left[\int_{s}^{T}|f(r,X_r,0,0)|^2\dx r\bigg|\mathcal{F}_s\right]\right)^{\frac{1}{2}}+$$
$$ +\sqrt{T-s}\left(
\left(\mathbb{E}\left[\int_{s}^{T}|f(r,X_r,\chi_H(Y_r,Z_r))-f(r,X_r,0,0)|^2\dx r\bigg|\mathcal{F}_s\right]\right)^{\frac{1}{2}}\right).$$
This leads to the bound $$\|Y\|_\infty\leq \check{C}+C_H\sqrt{T-t},$$ where only the second constant $C_H$ depends on the cutoff. $\check{C}$ is a function of $\|\xi\|_\infty$ and $\|f(\cdot,\cdot,0,0)\|_\infty$. \\

Now we only need to
\begin{itemize}
\item choose $H$ large enough such that $C_\sigma\cdot L_{\xi,x}$ and $\check{C}$ are both smaller $\frac{H}{2}$,
\item and then in the second step choose $t$ close enough to $T$, such that
$$ C_\sigma C_H(T-t)^{\frac{1}{4}}\textrm{ and }C_H\sqrt{T-t} $$
become smaller than $\frac{H}{2}$.
\end{itemize}
This means that $Y_s$ and $Z_s$, $s\in[t,T]$, a.e. do not leave the region in which the cutoff is "passive". \\
Therefore, $u$ is a decoupling field to $\fbsde (\xi,(\mu,\sigma,f))$, not just to $\fbsde (\xi,(\mu_H,\sigma_H,f_H))$.\\
Note furthermore that $u$ is unique since we only consider decoupling fields $\tilde{u}$ that are bounded and Lipschitz in $x$ (with $L_{\tilde{u},x}<L_{\sigma,z}^{-1}$), which means that we can always choose a cuttoff which is passive for $\tilde{u}$ \emph{and} $u$, allowing us to apply uniqueness from Corollary \ref{uniq}. \\
Strong regularity also follows from Theorem \ref{locexist}. \\
Similarly $u(s,x)$ must also be a.s. constant for all fixed $s\in[t,T]$, $x\in \mathbb{R}^n$.
\end{proof}

\begin{rem}\label{hchoiceM}
We observe from the proof that the supremum of all $h=T-t$ with $t$ satisfying the hypotheses of Theorem \ref{LocLip} can be bounded away from $0$ by a bound, which only depends on
\begin{itemize}
\item $L_{\xi,x}$, $L_{\xi,x}\cdot \|\sigma\|_{\infty}$ and $L_{\xi,x}\cdot L_{\sigma,z}$,
\item $\|\xi\|_{\infty}$ and $\|f(\cdot,\cdot,0,0)\|_\infty$,
\item the values $(L_H)_{H\in[0,\infty)}$ where $L_H$ is the Lipschitz constant of $\mu,\sigma$ and $f$ on $[0,T]\times\mathbb{R}^n\times B^1_H \times B^2_H$ w.r.t. to the last $3$ components, where $B^1_H\subset \mathbb{R}^{m}$ and $B^2_H\subset \mathbb{R}^{m\times d}$ are balls of radius $H$ with center $0$
\end{itemize}
and which is monotonically decreasing in these values.
\end{rem}

The following natural concept introduces a class of decoupling fields for non-Lipschitz problems (non-Lipschitz in $z$), to which nevertheless standard Lipschitz results can be applied.

\begin{defi}
Let $u$ be a decoupling field for $\fbsde(\xi,(\mu,\sigma,f))$. We call $u$ \emph{controlled in $z$} if there exists a constant $C>0$ such that for all $t_1,t_2\in[t,T]$, $t_1\leq t_2$, and all initial values $X_{t_1}=x\in\mathbb{R}^n$, the corresponding processes $X,Y,Z$ from the definition of a decoupling field satisfy $|Z_s(\omega)|\leq C$,
for almost all $(s,\omega)\in[t,T]\times\Omega$. If for a fixed triplet $(t_1,t_2,X_{t_1})$ there are different choices for $X,Y,Z$, then all of them are supposed to satisfy the above control.

We say that a decoupling field on $[t,T]$ is \emph{controlled in $z$} on a subinterval $[t_1,t_2]\subseteq[t,T]$ if $u$ restricted to $[t_1,t_2]$ is a decoupling field for $\fbsde(u(t_2,\cdot),(\mu,\sigma,f))$ that is controlled in $z$.

Furthermore we call a decoupling field on an interval $(s,T]$ \emph{controlled in $z$} if it is controlled in $z$ on every compact subinterval $[t,T]\subseteq (s,T]$ (with $C$ possibly depending on $t$).
\end{defi}

\begin{rem}
Our decoupling field from Theorem \ref{LocLip} is obviously controlled in $z$.
\end{rem}

\begin{rem}
Let $\mu,\sigma,f,\xi$ be as in Theorem \ref{LocLip} and assume that we have a decoupling field $u$ on some interval $[t,T]$, which is strongly regular, bounded and controlled in $z$. Then $u$ is also a solution to a Lipschitz problem obtained through a cutoff as in Theorem \ref{LocLip}. Furthermore, Lemma \ref{deter} is applicable since $u$ is a \emph{unique} strongly regular decoupling field (to a Lipschitz problem) according to Corollary \ref{uniq}. So $u$ is \underline{deterministic}. But now Lemma \ref{contin} is also applicable since due to the use of a cutoff we can assume Lipschitz continuity and thereby linear growth. So $u$ will also be \underline{continuous}.
\end{rem}

\begin{lem}\label{gluecontrolM}
Let $g,\mu,\sigma,f$ be as in Theorem \ref{LocLip}. For $0\leq s<t<T$ let $u$ be a bounded and weakly regular decoupling field for $\fbsde (\xi,(\mu,\sigma,f))$ on $[s,T]$.\\
If $u$ is controlled in $z$ on $[s,t]$ and $T-t$ is small enough as required in Theorem \ref{LocLip} resp. Remark \ref{hchoiceM} then $u$ is controlled in $z$ on $[s,T]$.
\end{lem}
\begin{proof}
Clearly $u$ is controlled in $z$ on $[s,t]$ and also on $[t,T]$ (with some constants). Define $C$ as the maximum of these constants. \\
We only need to control $Z$ for the case $s\leq t_1\leq t\leq t_2\leq T$, the other two cases being trivial. \\
The processes $X,Y,Z$ corresponding to the interval $[t_1,t_2]$ with an initial value $X_{t_1}$ will satisfy $|Z_r|\leq C$ for $r\in[t_1,t]$. At the same time, if we restrict $X,Y,Z$ to $[t,t_2]$, we observe that these restrictions satisfy the forward equation, backward equation and the decoupling condition for the interval $[t,t_2]$ with $X_t$ as initial value. Therefore $|Z_r|\leq C$ also holds for $r\in [t,t_2]$.
\end{proof}

As a consequence we can inductively show that sufficiently regular decoupling fields must be controlled in $z$.

\begin{cor}\label{controllM}
Let $\mu,\sigma,f,\xi$ be as in Theorem \ref{LocLip}. Assume that there exists a bounded and weakly regular decoupling field $u$ to this problem on some interval $[t,T]$. Then $u$ is controlled in $z$.
\end{cor}
\begin{proof}
Let $S\subseteq [t,T]$ be the set of all times $s\in[t,T]$, s.t. $u$ is controlled in $z$ on $[t,s]$.
\begin{itemize}
\item If $t\in S$, then $Z=0$, $Y=u(t,X_t)$.
\item Let $s\in S$ be arbitrary. According to Lemma \ref{gluecontrolM} there exists an $h>0$ s.t. $u$ is controlled in $z$ on $[t,(s+h)\wedge T]$ since  $\|u((s+h)\wedge T,\cdot)\|_\infty<\infty$, $L_{u((s+h)\wedge T,\cdot)}<\infty$ and $L_{\sigma,z}L_{u((s+h)\wedge T,\cdot)}<1$. Considering Remark \ref{hchoiceM} and the requirements  $\|u\|_\infty<\infty$, $L_{u,x}<\infty$ and $L_{\sigma,z}L_{u,x}<1$, we can choose $h$ independently of $s$.
\end{itemize}
This shows $S=[t,T]$ using small interval induction.
\end{proof}

The property of a decoupling field to be bounded and controlled in $z$ allows us to show the following two results as simple consequences of the theory developed in the Lipschitz case.

\begin{cor}\label{uniqM}
Let $\mu,\sigma,f,\xi$ be as in Theorem \ref{LocLip}. Assume that there are two bounded and weakly regular decoupling fields $u^{(1)},u^{(2)}$ to this problem on some interval $[t,T]$.
Then $u^{(1)}=u^{(2)}$.
\end{cor}
\begin{proof}
We know that $u^{(1)}$ and $u^{(2)}$ are controlled in $z$. Choose a passive cutoff (see proof of Theorem \ref{LocLip}) and apply Corollary \ref{uniq}.
\end{proof}

\begin{cor}\label{regulM}
Let $\mu,\sigma,f,\xi$ be as in Theorem \ref{LocLip}. Assume that there exists a bounded and weakly regular decoupling field $u$ to the corresponding FBSDE on some interval $[t,T]$. Then $u$ is strongly regular.
\end{cor}
\begin{proof}
$u$ is controlled in $z$. Choose a passive cutoff (see proof of Theorem \ref{LocLip}) and apply Corollary \ref{regul}.
\end{proof}

Remember the definition of the maximal interval. We aim at working with bounded decoupling fields. But $\mu,\sigma,f$ may depend in a super-linear way on $y$, such that singularities may very well occur because of exploding $u$ rather than exploding $u_x$. We therefore need to define a new type of maximal interval.

\begin{defi}
Let $I^b_{\mathrm{max}}\subseteq[0,T]$ for $\fbsde(\xi,(\mu,\sigma,f))$ be the union of all intervals $[t,T]\subseteq[0,T]$ such that there exists a bounded and weakly regular decoupling field $u$ on $[t,T]$.
\end{defi}

Note that this definition only makes sense if $\xi$ is bounded.

According to the following theorem, we have existence and uniqueness on $I^b_{\mathrm{max}}$.

\begin{thm}\label{globalexistM}
Let $\mu,\sigma,f,\xi$ be as in Theorem \ref{LocLip}. Then there exists a unique locally bounded and weakly regular decoupling field $u$ on $I^b_{\mathrm{max}}$. This $u$ is also controlled in $z$ and strongly regular. \\
Furthermore either $I^b_{\mathrm{max}}=[0,T]$ or $I^b_{\mathrm{max}}=(t^b_{\mathrm{min}},T]$, where $0\leq t^b_{\mathrm{min}}<T$.
\end{thm}
\begin{proof}
Let $t\in I^b_{\mathrm{max}}$. Obviously there exists a decoupling field $\check{u}^{(t)}$ on $[t,T]$ satisfying $L_{\check{u}^{(t)},x}<L_{\sigma,z}^{-1}$ and $\|\check{u}^{(t)}\|_\infty<\infty$. There is only one such $\check{u}^{(t)}$ according to Corollary \ref{uniqM}. Furthermore, for $t,t'\in I^b_{\mathrm{max}}$ the functions $\check{u}^{(t)}$ and $\check{u}^{(t')}$ coincide on $[t\vee t',T]$ because of Corollary \ref{uniqM}. \\
Define $u(t,\cdot):=\check{u}^{(t)}(t,\cdot)$ for all $t\in I^b_{\mathrm{max}}$. This function $u$ is a decoupling field on $[t,T]$, since it coincides with $\check{u}^{(t)}$ on $[t,T]$. Therefore $u$ is a decoupling field on the whole interval $I^b_{\mathrm{max}}$ and satisfies $L_{u|_{[t,T]},x}<L_{\sigma,z}^{-1}$, $\|u|_{[t,T]}\|_\infty<\infty$ for all $t\in I^b_{\mathrm{max}}$. \\
Uniqueness of $u$ follows directly from Corollary \ref{uniqM} applied to every interval $[t,T]\subseteq I^b_{\mathrm{max}}$. \\
Furthermore $u$ is controlled in $Z$ and strongly regular on $[t,T]$ for all $t\in I^b_{\mathrm{max}}$ due to Corollaries \ref{controllM} and \ref{regulM}. \\
Addressing the form of $I^b_{\mathrm{max}}$, we see that $I^b_{\mathrm{max}}=[t,T]$ with $t\in(0,T]$ is not possible. Assume otherwise. According to Theorem \ref{globalexistM} there exists a decoupling field $u$ on $[t,T]$ s.t. $L_{u,x}<L_{\sigma,z}^{-1}$ and $\|u\|_\infty<\infty$. But then $u$ can be extended a little bit to the left using Theorem \ref{LocLip} (and Lemma \ref{glue}).
\end{proof}

The following result basically states that for a singularity to occur either $u$ or $u_x$ has to "explode" at $t_{\mathrm{min}}$.

\begin{lem}\label{explosionM}
Let $\mu,\sigma,f,\xi$ be as in Theorem \ref{LocLip}. If $I^b_{\mathrm{max}}=(t^b_{\mathrm{min}},T]$, then
$$\lim_{t\downarrow t_{\mathrm{min}}}\left(\left(1+L_{u(t,\cdot),x}\right)^{-1}-\left(1+L_{\sigma,z}^{-1}\right)^{-1}\right)\left(\|u(t,\cdot)\|_\infty+1\right)^{-1}=0,$$
where $u$ is the decoupling field according to Theorem \ref{globalexistM}.
\end{lem}
\begin{proof}
We argue indirectly. Assume otherwise. Then we can select times $t_n\downarrow t^b_{\mathrm{min}}$, $n\rightarrow\infty$ such that
$$\sup_{n\in\mathbb{N}}L_{u(t_n,\cdot),x}<L_{\sigma,z}^{-1} \quad\textrm{and}\quad \sup_{n\in\mathbb{N}}\|u(t_n,\cdot)\|_\infty<\infty.$$
But then we may choose an $h>0$ according to Remark \ref{hchoiceM} which does not depend on $n$ and then choose $n$ large enough to have $t_n-t^b_{\mathrm{min}}<h$. So $u$ can be extended to the left to a larger interval $[(t_n-h)\vee 0,T]$ contradicting the definition of $I^b_{\mathrm{max}}$.
\end{proof}

\section{Appendix}

\begin{lem}\label{fundulem}
Let $X:\mathcal{M}\times\Lambda\rightarrow\mathbb{R}$ be a weakly differentiable mapping, where $(\mathcal{M},\mathcal{A},\rho)$  is some complete measure space and $\Lambda\subseteq\mathbb{R}^N$ is open, $N\in\mathbb{N}$.\\
Let $h\in(0,\infty)$, $v\in \mathbb{R}^N$ and $\Lambda^h$ the open set of all $\lambda\in\Lambda$ such that $\overline{B_{h}(\lambda)}\subseteq\Lambda$. Then
$$\int_0^h\frac{\dx}{\dx\lambda}X(\omega,\lambda_0+t v)v\dx t=X(\omega,\lambda_0+hv)-X(\omega,\lambda_0)$$
for almost all $\lambda_0\in\Lambda^{h|v|}$, for almost all $\omega\in\mathcal{M}$.
\end{lem}
\begin{proof}
Choose an $\omega\in\Omega$ s.t. $X(\omega,\cdot)$ is weakly differentiable with a weak derivative $\frac{\dx}{\dx\lambda}X(\omega,\cdot)$.
Define a map $F:\Lambda^{h|v|}\rightarrow\mathbb{R}$ via
$$F(\lambda_0):=\int_0^h\frac{\dx}{\dx\lambda}X(\omega,\lambda_0+t v)v\dx t-\left(X(\omega,\lambda_0+hv)-X(\omega,\lambda_0)\right).$$
Note here that $(t,\lambda_0)\mapsto\frac{\dx}{\dx\lambda}X(\omega,\lambda_0+t v)$ is locally integrable since
$$
\int_0^h \int_{\Lambda^{h|v|}\cap K}\left|\frac{\dx}{\dx\lambda}X(\omega,\lambda_0+t v)\right|\dx\lambda_0\dx t\leq \int_0^h \int_{\Lambda\cap \overline{B_{h}(K)}}\left|\frac{\dx}{\dx\lambda}X(\omega,\lambda_0)\right|\dx\lambda_0\dx t= $$
$$ =h\int_{\Lambda\cap \overline{B_{h}(K)}}\left|\frac{\dx}{\dx\lambda}X(\omega,\lambda_0)\right|\dx\lambda_0<\infty, $$
for compacts $K\subseteq \mathbb{R}^N$. The closed $h$ - neighborhood $\overline{B_{h}(K)}$ of $K$ is also compact obviously.\\
We want to show that $F$ is a.e. equal $0$. For this purpose it is sufficient to show $\int_{\Lambda^{h|v|}}F(\lambda_0)\varphi(\lambda_0)\dx\lambda_0=0$ for every test function $\varphi\in C^{\infty}_c(\Lambda^{h|v|})$. Now let $\varphi$ be such a function. Then
$$ \int_{\Lambda^{h|v|}}\int_0^h\varphi(\lambda_0)\frac{\dx}{\dx\lambda}X(\omega,\lambda_0+t v)v\dx t\dx\lambda_0=\int_0^h \int_{\Lambda^{h|v|}}\varphi(\lambda_0)\frac{\dx}{\dx\lambda}X(\omega,\lambda_0+t v)\dx\lambda_0v\dx t=$$
$$=\int_0^h \int_{\Lambda^{h|v|}+t v}\varphi(\lambda_0-tv)\frac{\dx}{\dx\lambda}X(\omega,\lambda_0)\dx\lambda_0v\dx t=\int_0^h \int_{\Lambda}\varphi(\lambda_0-tv)\frac{\dx}{\dx\lambda}X(\omega,\lambda_0)\dx\lambda_0v\dx t= $$
$$=-\int_0^h \int_{\Lambda}X(\omega,\lambda_0)\frac{\dx}{\dx\lambda}\varphi(\lambda_0-tv)\dx\lambda_0v\dx t=-\int_{\Lambda}\int_0^h X(\omega,\lambda_0)\frac{\dx}{\dx\lambda}\varphi(\lambda_0-tv)v\dx t\dx\lambda_0=$$
$$=\int_{\Lambda}X(\omega,\lambda_0)\int_0^h \frac{\dx}{\dx\lambda}\varphi(\lambda_0-tv)(-v)\dx t\dx\lambda_0=\int_{\Lambda}X(\omega,\lambda_0)\left(\varphi(\lambda_0-hv)-\varphi(\lambda_0)\right)\dx\lambda_0=$$
$$=\int_{\Lambda}X(\omega,\lambda_0)\varphi(\lambda_0-hv)\dx\lambda_0-\int_{\Lambda}X(\omega,\lambda_0)\varphi(\lambda_0)\dx\lambda_0= $$
$$=\int_{\Lambda^{h|v|}+hv}X(\omega,\lambda_0)\varphi(\lambda_0-hv)\dx\lambda_0-\int_{\Lambda^{h|v|}}X(\omega,\lambda_0)\varphi(\lambda_0)\dx\lambda_0= $$
$$=\int_{\Lambda^{h|v|}}X(\omega,\lambda_0+hv)\varphi(\lambda_0)\dx\lambda_0-\int_{\Lambda^{h|v|}}X(\omega,\lambda_0)\varphi(\lambda_0)\dx\lambda_0= $$
$$=\int_{\Lambda^{h|v|}}\left(X(\omega,\lambda_0+hv)-X(\omega,\lambda_0)\right)\varphi(\lambda_0)\dx\lambda_0. $$
This already implies $\int_{\Lambda^{h|v|}}F(\lambda_0)\varphi(\lambda_0)\dx\lambda_0=0$.
\end{proof}

\begin{proof}[Proof of Lemma \ref{chainrule}]
For the weak differentiability of $g(X)$ consult \cite{ambrosio}, Corollary 3.2 (applied $\omega$-wise). \\
We use the following notation. For some vector $(x_1,\ldots,x_k)\in\mathbb{R}^{\sum_{i=1}^kd_i}$, $x^{i,j}$ refers to the vector $(x_i,\ldots,x_j)$, where $i<j$, $i,j\in\{1,\ldots,k\}$. \\
Let $v\in S^{n-1}$ be fixed. Let $\lambda\in\mathbb{R}^n$ and $t\in\mathbb{R}$. Then
\begin{multline}\label{telescope} g(\omega,X(\omega,\lambda+tv))-g(\omega,X(\omega,\lambda))= \\
=\sum_{i=1}^k\left(g\left(\omega,X^{1,i}(\omega,\lambda+tv),X^{i+1,k}(\omega,\lambda)\right)-g\left(\omega,X^{1,i-1}(\omega,\lambda+tv),X^{i,k}(\omega,\lambda)\right)\right) = \\
=\sum_{i=1}^k\frac{g\left(X^{1,i}(\lambda+tv),X^{i+1,k}(\lambda)\right)-g\left(X^{1,i-1}(\lambda+tv),X^{i,k}(\lambda)\right)}{X_i(\lambda+tv)-X_i(\lambda)}(\omega) \left(X_i(\omega,\lambda+tv)-X_i(\omega,\lambda)\right),
\end{multline}
where we use the convention $\frac{0}{0}:=0$. Now define
$$ \Delta^{X}_{x_i}g(\cdot,\lambda,v):=\limsup_{l\rightarrow\infty}\frac{g\left(X^{1,i}(\lambda+\frac{1}{l}v),X^{i+1,k}(\lambda)\right)-g\left(X^{1,i-1}(\lambda+\frac{1}{l}v),X^{i,k}(\lambda)\right)}{X_i(\lambda+\frac{1}{l}v)-X_i(\lambda)}. $$
Note that $|\Delta^{X}_{x_i}g|$ can be assumed to be bounded by $L_{g,x_i}$ (everywhere), due to Lipschitz continuity of $g$ in the $i$-th component (we can assume without loss of generality, that $g$ is Lipschitz for all $\omega$). Furthermore $\Delta^{X}_{x_i}g$ is clearly measurable, as a $\limsup$ of measurable mappings.
Also note that for almost all $\lambda,\omega$
$$ \lim_{t\rightarrow 0}\frac{g(\omega,X(\omega,\lambda+tv))-g(\omega,X(\omega,\lambda))}{t}=\left(\frac{\dx}{\dx \lambda}g(X)(\omega,\lambda)\right)v, $$
$$ \lim_{t\rightarrow 0}\frac{X_i(\omega,\lambda+tv)-X_i(\omega,\lambda)}{t}=\left(\frac{\dx}{\dx \lambda}X_i(\omega,\lambda)\right)v. $$
(Recall that $v$ is fixed here.) \\
Now for all such $\lambda,\omega$ we can choose a sequence $(t_q)_{q\in\mathbb{N}}$, which is a subsequence of $\left(\frac{1}{l}\right)_{l\in\mathbb{N}}$ and such that
$$ \Delta^{X}_{x_i}g(\omega,\lambda,v)=\lim_{q\rightarrow \infty}\frac{g\left(X^{1,i}(\lambda+t_qv),X^{i+1,k}(\lambda)\right)-g\left(X^{1,i-1}(\lambda+t_qv),X^{i,k}(\lambda)\right)}{X_i(\lambda+t_qv)-X_i(\lambda)}(\omega). $$
At the same time
$$ \lim_{q\rightarrow \infty}\frac{g(\omega,X(\omega,\lambda+t_qv))-g(\omega,X(\omega,\lambda))}{t_q}=\left(\frac{\dx}{\dx \lambda}g(X)(\omega,\lambda)\right)v, $$
$$ \lim_{q\rightarrow \infty}\frac{X_i(\omega,\lambda+t_qv)-X_i(\omega,\lambda)}{t_q}=\left(\frac{\dx}{\dx \lambda}X_i(\omega,\lambda)\right)v.$$
So we obtain from (\ref{telescope})
$$ \left(\frac{\dx}{\dx \lambda}g(X)(\omega,\lambda)\right)v=\sum_{i=1}^k\left(\Delta^{X}_{x_i}g(\omega,\lambda,v)\right)\left(\frac{\dx}{\dx \lambda}X_i(\omega,\lambda)\right)v.$$
\end{proof}

\begin{proof}[Proof of Lemma \ref{wd1}]
Define mappings $Y:\Omega\times\mathbb{R}^n\rightarrow\mathbb{R}$ and $Z:\Omega\times\mathbb{R}^n\rightarrow\mathbb{R}^n$ via
$Y(\omega,\lambda):=\mathbb{E}\left[X(\cdot,\lambda)|\mathcal{G}\right](\omega)$ and $Z(\omega,\lambda):=\mathbb{E}\left[\frac{\dx}{\dx \lambda}X(\cdot,\lambda)|\mathcal{G}\right](\omega)$.
Note that $Y$ is measurable, since $Y\mathbf{1}_{\Omega\times B_R(\lambda_0)}$ can be seen as conditional expectation w.r.t. $\mathcal{G}\otimes\mathcal{L}(B_R(\lambda_0))$ on the space
$$\left(\Omega\times B_R(\lambda_0),\,\mathcal{F}\otimes\mathcal{L}(B_R(\lambda_0)),\,\mathbb{P}\otimes \frac{1}{|B_R(\lambda_0)|}\rho_{B_R(\lambda_0)}\right),$$ where $\lambda_0\in\mathbb{R}^n$ and $R>0$ are arbitrary. \\
We now claim that $Z$ is the weak derivative of $Y$. To see this, take a test function $\varphi\in C^\infty_c(\mathbb{R}^n)$. We have
$$\int_{\mathbb{R}^n}\varphi(\lambda)Y(\cdot,\lambda)\dx\lambda=\int_{\mathbb{R}^n}\varphi(\lambda)\mathbb{E}\left[\frac{\dx}{\dx \lambda}X(\cdot,\lambda)\Big|\mathcal{G}\right]\dx\lambda=
\mathbb{E}\left[\int_{\mathbb{R}^n}\varphi(\lambda)\frac{\dx}{\dx \lambda}X(\cdot,\lambda)\dx\lambda\bigg|\mathcal{G}\right]=$$
$$=\mathbb{E}\left[\int_{\mathbb{R}^n}\frac{\dx}{\dx \lambda}\varphi(\lambda)X(\cdot,\lambda)\dx\lambda\bigg|\mathcal{G}\right]=
\int_{\mathbb{R}^n}\frac{\dx}{\dx \lambda}\varphi(\lambda)\mathbb{E}\left[X(\cdot,\lambda)\Big|\mathcal{G}\right]\dx\lambda=\int_{\mathbb{R}^n}\frac{\dx}{\dx \lambda}\varphi(\lambda)Z(\cdot,\lambda)\dx\lambda,$$
where we used Fubini's theorem.
\end{proof}

\begin{proof}[Proof of Lemma \ref{wd3}]
Measurability of $X$ follows from the fact that the stochastic integral $\int_0^TZ_s\dx W_s$ can be defined as an a.e. limit of integrals over simple progressive processes $Z^n$ and such integrals are measurable, since $Z^n_s(\cdot,\cdot)$ must be measurable for every $s\in[0,T]$.\\
It remains to verify that $\int_0^T\frac{\dx}{\dx\lambda}Z_s(\cdot,\lambda)^\top\dx W_s$ is a weak derivative of $X$. To see this, take a test function $\varphi~\in~ C^\infty_c(\mathbb{R}^n)$ and choose any index $i\in\{1,\ldots,n\}$. Then
$$\int_{\mathbb{R}^n}\varphi(\lambda)\int_0^T\frac{\dx}{\dx\lambda_i}Z_s(\cdot,\lambda)^\top\dx W_s\dx\lambda=\int_{\mathbb{R}^n}\int_0^T\varphi(\lambda)\frac{\dx}{\dx\lambda_i}Z_s(\cdot,\lambda)^\top\dx W_s\dx\lambda= $$
$$ =\int_0^T\int_{\mathbb{R}^n}\varphi(\lambda)\frac{\dx}{\dx\lambda_i}Z_s(\cdot,\lambda)^\top\dx\lambda\dx W_s=\int_0^T\int_{\mathbb{R}^n}\frac{\dx}{\dx\lambda_i}\varphi(\lambda)Z_s(\cdot,\lambda)^\top\dx\lambda\dx W_s=$$
$$=\int_{\mathbb{R}^n}\frac{\dx}{\dx\lambda_i}\varphi(\lambda)\int_0^TZ_s(\cdot,\lambda)^\top\dx W_s\dx\lambda=\int_{\mathbb{R}^n}\frac{\dx}{\dx\lambda_i}\varphi(\lambda)X(\cdot,\lambda)\dx\lambda, $$
where used continuity and linearity of the stochastic integral twice.
\end{proof}

\begin{proof}[Proof of Lemma \ref{wd4}]
Existence of $Z$ follows from the It\^o representation formula, which is applied to $X$. It can also be applied $\frac{\dx}{\dx\lambda}X$, yielding a second progressively measurable process $\tilde{Z}:\Omega\times[0,T]\times\mathbb{R}^n\rightarrow\mathbb{R}^{d\times n}$. It remains to show that $\tilde{Z}$ is a weak derivative of $Z$. To verify this, take a test function $\varphi~\in~ C^\infty_c(\mathbb{R}^n)$ and choose any index $i\in\{1,\ldots,n\}$. Then we have
$$ \int_0^T\int_{\mathbb{R}^n}\varphi(\lambda)\tilde{Z}^i_s(\cdot,\lambda)^\top\dx\lambda\dx W_s=\int_{\mathbb{R}^n}\varphi(\lambda) \int_0^T\tilde{Z}^i_s(\cdot,\lambda)^\top\dx W_s\dx\lambda=$$
$$ =\int_{\mathbb{R}^n}\varphi(\lambda)\left(\frac{\dx}{\dx\lambda_i}X(\cdot,\lambda)-\mathbb{E}\left[\frac{\dx}{\dx\lambda_i}X(\cdot,\lambda)\right]\right)\dx\lambda=
\int_{\mathbb{R}^n}\frac{\dx}{\dx\lambda_i}\varphi(\lambda)\left(X(\cdot,\lambda)-\mathbb{E}\left[X(\cdot,\lambda)\right]\right)\dx\lambda=$$
$$ =\int_{\mathbb{R}^n}\frac{\dx}{\dx\lambda_i}\varphi(\lambda)\int_0^TZ_s(\cdot,\lambda)^\top\dx W_s\dx\lambda=\int_0^T\int_{\mathbb{R}^n}\frac{\dx}{\dx\lambda_i}\varphi(\lambda)Z_s(\cdot,\lambda)^\top\dx\lambda\dx W_s ,$$
which implies $\int_{\mathbb{R}^n}\varphi(\lambda)\tilde{Z}^i_s(\cdot,\lambda)^\top\dx\lambda=\int_{\mathbb{R}^n}\frac{\dx}{\dx\lambda_i}\varphi(\lambda)Z_s(\cdot,\lambda)^\top\dx\lambda$.
\end{proof}

\begin{proof}[Proof of Lemma \ref{wd5}]
For each $t\in [0,T]$ define
$$ Y_t:=\mathbb{E}\left[X-\int_t^TV_s\dx s\bigg|\mathcal{F}_t\right]=\mathbb{E}\left[X-\int_0^TV_s\dx s\bigg|\mathcal{F}_t\right]+\int_0^tV_s\dx s . $$
The mapping $Y_t:\Omega\times\mathbb{R}^n\rightarrow\mathbb{R}$ is $\mathcal{F}_t\otimes\mathcal{L}(\mathbb{R}^n)$ measurable and weakly differentiable w.r.t. $\lambda$, such that
$\frac{\dx}{\dx\lambda}Y_t=\mathbb{E}\left[\frac{\dx}{\dx\lambda}X-\int_t^T\frac{\dx}{\dx\lambda}V_s\dx s\bigg|\mathcal{F}_t\right]$, according to Lemma \ref{wd2} and  Lemma \ref{wd1}. Thereby we obtain a process $Y$, which is continuous in time and therefore progressively measurable. Now define
$$ M:=X-\int_0^TV_s\dx s-Y_0 .$$
$M:\Omega\times\mathbb{R}^n\rightarrow\mathbb{R}$ is measurable and weakly differentiable w.r.t. $\lambda$. It is also straightforward to check, that $\mathbb{E}\left[\left|M(\cdot,\lambda)\right|^2\right]$ and $\mathbb{E}\left[\left|\frac{\dx}{\dx\lambda}M(\cdot,\lambda)\right|^2\right]$ are both locally integrable w.r.t. $\lambda$. Therefore we can apply Lemma \ref{wd4} and write $M=\int_0^TZ_s^\top\dx W_s$ with a progressively measurable and weakly differentiable $Z$. Also
$$ \int_0^TZ_s^\top\dx W_s=X-\int_0^TV_s\dx s-Y_0 .$$
Applying conditional expectations yields
$$ \int_0^tZ_s^\top\dx W_s=\mathbb{E}\left[X-\int_0^TV_s\dx s\bigg|\mathcal{F}_t\right]-Y_0=Y_t-\int_0^tV_s\dx s-Y_0 .$$
Subtracting this equation from the preceding one leads to
$$ \int_t^TZ_s^\top\dx W_s=X-\int_t^TV_s\dx s-Y_t .$$
We can now differentiate w.r.t. $\lambda$ according to Lemma \ref{wd3} with the result
$$ \int_t^T\frac{\dx}{\dx\lambda}Z_s^\top\dx W_s=\frac{\dx}{\dx\lambda}X-\int_t^T\frac{\dx}{\dx\lambda}V_s\dx s-\frac{\dx}{\dx\lambda}Y_t .$$
\end{proof}

\begin{proof}[Proof of Lemma \ref{weakderivexist}]
Let $K>0$ be a constant, and let $\Lambda_R:=\Lambda\cap B_R(0)$, $R\in\mathbb{N},$ be a bounded subset of $\Lambda$. Then
$\int_{\mathcal{M}}|X_i(\lambda,\cdot)-X(\lambda,\cdot)|^2\wedge K\dx\rho\rightarrow 0$ for $i\rightarrow\infty$ for almost all $\lambda\in\Lambda$. Thus
$$\int_{\Lambda_R}\int_{\mathcal{M}}|X_i(\lambda,\cdot)-X(\lambda,\cdot)|^2\wedge K\dx\rho\dx\lambda\rightarrow 0$$ for $i\rightarrow\infty$ by dominated convergence. Thus $(X_i)$ is a Cauchy sequence in measure on $\Lambda_R\times\mathcal{M}$. Passing to a subsequence we can assume that $(X_i)$ converges almost everywhere to some measurable $\tilde{X}_R$ on $\Lambda_R\times\mathcal{M}$. Since $R\in\mathbb{N}$ can be chosen arbitrarily, we can, again by passing to a subsequence, assume that $(X_i)$ converges almost everywhere to some measurable $\tilde{X}$ on $\Lambda\times\mathcal{M}$. It is easy to check using uniqueness of $\mathcal{L}^2$-limits that $X$ and $\tilde{X}$ must coincide a.e.. Since $\Lambda\times\mathcal{M}$ is a complete measure space, $X$ is measurable. \\
Next we claim that $\int_{\mathcal{M}}|X_i(\lambda,\cdot)|^2\dx\rho$ is bounded (independent of $i$ and $\lambda$) on balls of the form $B_\varepsilon(\lambda_0)\subseteq \Lambda$ with $0<\varepsilon\leq 1$. We have for a.a. $\lambda\in B_\varepsilon(\lambda_0)$
$$ \left|\sqrt{\int_{\mathcal{M}}|X_i(\lambda,\cdot)|^2\dx\rho}-
\sqrt{\int_{\mathcal{M}}|X_i(\lambda_0,\cdot)|^2\dx\rho}\right|
\leq \sqrt{\int_{\mathcal{M}}|X_i(\lambda,\cdot)-X_i(\lambda_0,\cdot)|^2\dx\rho}=$$
$$= \sqrt{\int_{\mathcal{M}}\left|\int_0^1 \frac{\dx}{\dx\lambda}X_i(\lambda_0+s(\lambda-\lambda_0),\cdot)(\lambda-\lambda_0)\dx s\right|^2\dx\rho}\leq $$
$$\leq \sqrt{\int_{\mathcal{M}}\int_0^1\left| \frac{\dx}{\dx\lambda}X_i(\lambda_0+s(\lambda-\lambda_0),\cdot)(\lambda-\lambda_0)\right|^2\dx s\dx\rho}\leq $$
$$\leq \sqrt{\int_0^1\int_{\mathcal{M}}\left| \frac{\dx}{\dx\lambda}X_i(\lambda_0+s(\lambda-\lambda_0),\cdot)\right|^2\left|\lambda-\lambda_0\right|^2\dx\rho\dx s}\leq \sqrt{C} \left|\lambda-\lambda_0\right|\leq \sqrt{C}.$$
Therefore $(i,\lambda)\mapsto\int_{\mathcal{M}}|X_i(\lambda,\cdot)|^2\dx\rho$, $\lambda\in B_\varepsilon(\lambda_0)$, $i\in\mathbb{N},$ must be bounded, since $\sup_{i\in\mathbb{N}}\int_{\mathcal{M}}|X_i(\lambda_0,\cdot)|^2\dx\rho<\infty$, because of the $\mathcal{L}^2$--convergence of $X_i(\lambda_0,\cdot)$. \\
Let $\mathbb{H}_{\varepsilon,\lambda_0,\delta}$ be the space of real valued measurable functions $Y$ on $S_{\varepsilon,\lambda_0}:=B_\varepsilon(\lambda_0)\times\mathcal{M}$ s.t.
$$\sum_{0\leq |\alpha|\leq\delta}\int_{B_\varepsilon(\lambda_0)}\int_{\mathcal{M}}|D^\alpha_\lambda Y(\lambda,\cdot)|^2\dx\rho\dx\lambda < \infty.$$
Obviously $(X_i)$ is a bounded sequence in $\mathbb{H}_{\varepsilon,\lambda_0,\delta}$.
We claim that $X$ must be in $\mathbb{H}_{\varepsilon,\lambda_0,\delta}$ too.
Let $\alpha\in\mathbb{N}^N$ be a multiindex s.t. $1\leq|\alpha|\leq\delta$. We have
$$\int_{B_\varepsilon(\lambda_0)}\int_{\mathcal{M}}D^\alpha_\lambda X_i(\lambda,\cdot)\varphi(\lambda,\cdot)\dx\rho\dx\lambda=(-1)^{|\alpha|}
\int_{B_\varepsilon(\lambda_0)}\int_{\mathcal{M}}X_i(\lambda,\cdot)D^\alpha_\lambda \varphi(\lambda,\cdot)\dx\rho\dx\lambda $$
for all $\varphi\in\mathbb{H}_{\varepsilon,\lambda_0,\delta}$ s.t. the support of $\varphi(\cdot,\omega)$ is a subset of $B_\varepsilon(\lambda_0)$ for all $\omega\in\mathcal{M}$. $(D^\alpha_\lambda X_i)$ is a bounded sequence in the Hilbert space $\mathbb{H}_{\varepsilon,\lambda_0}=\mathcal{L}^2(S_{\varepsilon,\lambda_0})$. Therefore, by passing to a subsequence, we can assume that there exists a weak limit $X^\alpha$ in $\mathcal{L}^2(S_{\varepsilon,\lambda_0})$. Thus
$$ \lim_{i\rightarrow\infty}\int_{B_\varepsilon(\lambda_0)}\int_{\mathcal{M}}D^\alpha_\lambda X_i(\lambda,\cdot)\varphi(\lambda,\cdot)\dx\rho\dx\lambda=
\int_{B_\varepsilon(\lambda_0)}\int_{\mathcal{M}}X^\alpha(\lambda,\cdot)\varphi(\lambda,\cdot)\dx\rho\dx\lambda. $$
On the other hand
$$ \lim_{i\rightarrow\infty}\int_{B_\varepsilon(\lambda_0)}\int_{\mathcal{M}}X_i(\lambda,\cdot)D^\alpha_\lambda \varphi(\lambda,\cdot)\dx\rho\dx\lambda=
\int_{B_\varepsilon(\lambda_0)}\int_{\mathcal{M}}X(\lambda,\cdot)D^\alpha_\lambda \varphi(\lambda,\cdot)\dx\rho\dx\lambda $$
by the $\mathcal{L}^2$ convergence of the $X_i$. This shows the weak differentiability of $X$ w.r.t. $\lambda$ and $X^\alpha=D^\alpha_\lambda X$. In particular $X\in \mathbb{H}_{\varepsilon,\lambda_0,\delta}$.

It remains to show
$$g(\lambda):=\sum_{1\leq |\alpha|\leq\delta}\int_{\mathcal{M}}|D^\alpha_\lambda X(\lambda,\cdot)|^2\dx\rho \leq C,$$
for almost all $\lambda\in B_\varepsilon(\lambda_0)$. Let $B\subseteq B_\varepsilon(\lambda_0)$ be a measurable subset of $B_\varepsilon(\lambda_0)$. Using weak convergence we get
$$\int_{B}g(\lambda)\dx\lambda=\sum_{1\leq |\alpha|\leq\delta}\int_{B}\int_{\mathcal{M}}
|D^\alpha_\lambda X(\lambda,\cdot)|^2\dx\rho\dx\lambda=\lim_{i\rightarrow\infty}\sum_{1\leq |\alpha|\leq\delta}\int_{B}\int_{\mathcal{M}}
D^\alpha_\lambda X_i(\lambda,\cdot)D^\alpha_\lambda X(\lambda,\cdot)\dx\rho\dx\lambda. $$
Using Cauchy-Schwarz' inequality and $\sum_{1\leq |\alpha|\leq\delta}\int_{\mathcal{M}}|D_\lambda^\alpha X_i(\lambda,\cdot)|^2\dx\rho\leq C$, we obtain
$$ \int_{B}g(\lambda)\dx\lambda\leq \sqrt{|B|\cdot C}\cdot \left(\int_{B}g(\lambda)\dx\lambda\right)^{\frac{1}{2}}. $$
In other words $\frac{1}{|B|}\int_{B}g(\lambda)\dx\lambda\leq C$ for all measurable $B\subseteq B_\varepsilon(\lambda_0)$. This implies $g\leq C$ a.e. by Lebesgue's differentiation theorem.
\end{proof}

\end{document}